\documentclass[10pt,letterpaper]{article}
\usepackage[T1]{fontenc}

\usepackage[english]{babel}
\usepackage{geometry,amsmath,amssymb,graphicx,enumerate,latexsym,xcolor,theorem}
\usepackage{xcolor}
\definecolor{bl}{rgb}{0.0,0.2,0.6}

\usepackage{fancyhdr}
\usepackage[runin]{abstract}			
\abslabeldelim{}						%
\usepackage{bbold}
\normalfont
      plus .92\fontdimen3\font
     minus .92\fontdimen4\font

\newcommand{\assign}{:=}
\newcommand{\mathd}{\mathrm{d}}
\newcommand{\tmaffiliation}[1]{\\ #1}

\newcommand{\tmem}[1]{{\em #1\/}}
\newcommand{\tmmathbf}[1]{\ensuremath{\boldsymbol{#1}}}
\newcommand{\tmname}[1]{\textsc{#1}}
\newcommand{\tmop}[1]{\ensuremath{\operatorname{#1}}}

\newcommand{\tmsep}{, }
\newcommand{\tmstrong}[1]{\textbf{#1}}

\newcommand{\tmtextbf}[1]{{\bfseries{#1}}}
\newcommand{\tmtextit}[1]{{\itshape{#1}}}
\newcommand{\tmtextsc}[1]{{\scshape{#1}}}

\newcommand{\tmtextup}[1]{{\upshape{#1}}}
\newenvironment{enumeratenumeric}{\begin{enumerate}[1.] }{\end{enumerate}}
\newenvironment{itemizedot}{\begin{itemize} }{\end{itemize}}
\newenvironment{proof}{\noindent\textbf{Proof\ }}{\hspace*{\fill}$\Box$\medskip}
\newtheorem{definition}{Definition}
\newtheorem{lemma}{Lemma}
\newtheorem{proposition}{Proposition}
{\theorembodyfont{\rmfamily}\newtheorem{remark}{Remark}}
\newtheorem{theorem}{Theorem}
\newcommand{\tmkeywords}{\textbf{Keywords:} }
\newcommand{\tmmsc}{\textbf{A.M.S. subject classification:} }

\newcommand{\T}{\mathsf{T}}
\newcommand{\SStwo}{\mathbf{S}^2}
\newcommand{\SSn}{\mathbf{S}^{n - 1}}

\newcommand{\eq}{=}
\newcommand{\eqs}{\; = \;}
\newcommand{\mybold}[1]{\tmtextbf{\tmtextit{#1}}}
\newcommand{\twos}{\xrightarrow{2 s}}

\newcommand{\curl}{\tmmathbf{\tmop{curl}}\hspace{1pt}}
\newcommand{\curlx}{\tmmathbf{\tmop{curl}}_{\hspace{0.5pt} x}}

\newcommand{\divv}{\tmop{div}}
\newcommand{\divx}{\tmop{div}_{\hspace{-1pt} x}}
\newcommand{\divy}{\tmop{div}_{\hspace{-1pt} y}}
\newcommand{\grady}{\nabla_{\hspace{-1pt} y}}
\newcommand{\gradx}{\nabla_{\hspace{-1pt} x}}

\newcommand{\RR}{\mathbb{R}}

\newcommand{\NN}{\mathbb{N}}

\newcommand{\YY}{Y}
\newcommand{\YYe}{Y_{\varepsilon}}

\newcommand{\lo}[1]{\text{\tmtextbf{\tmtextup{#1}}}}

\makeatletter							
\def\printtitle{
    {\color{bl} \centering \Large \sc \textbf{\@title}\par}}		
\makeatother							
\usepackage{sectsty}
\allsectionsfont{\color{bl}\normalsize\selectfont}
\sectionfont{\color{bl} \centering  \sc \bfseries \large \selectfont}
\begin{document}

\title{
  \color{bl}\textbf{Cell averaging two-scale convergence}\\
  \large\textbf{Applications to periodic homogenization}
}

\author{
 \textsc{Fran{\c c}ois Alouges}
  \tmaffiliation{\normalsize CMAP, {\'E}cole Polytechnique,\\
  \normalsize route de Saclay,\\
  \normalsize 91128 Palaiseau Cedex,\\
  \normalsize France}
  \and
 \textsc{ Giovanni Di Fratta}
  \tmaffiliation{\normalsize CMAP, {\'E}cole Polytechnique,\\
  \normalsize route de Saclay,\\
  \normalsize 91128 Palaiseau Cedex,\\
  \normalsize France}
}

\maketitle

\begin{abstract}
  The aim of the paper is to introduce an alternative
  notion of two-scale convergence which gives a more natural modeling approach
  to the homogenization of partial differential equations with periodically
  oscillating coefficients: while removing the bother of the admissibility of
  test functions, it nevertheless simplifies the proof of all the standard
  compactness results which made classical two-scale convergence very worthy
  of interest: bounded sequences in $L^2_{\sharp} [ \YY, L^2 (\Omega)
  ]$ and $L^2_{\sharp} [ \YY, H^1 (\Omega) ]$ are proven to
  be relatively compact with respect to this new type of convergence. The
  strengths of the notion are highlighted on the classical homogenization
  problem of linear second-order elliptic equations for which first order
  boundary corrector-type results are also established. Eventually, possible
  weaknesses of the method are pointed out on a nonlinear problem: the weak
  two-scale compactness result for $\SStwo$-valued stationary harmonic maps.
\end{abstract}

\tmmsc{35B27}{\tmsep}{35B40}{\tmsep}{74Q05}

\tmkeywords{periodic homogenization}{\tmsep}{two-scale
convergence}{\tmsep}{boundary layers}{\tmsep}{cell
averaging}{\tmsep}{multiscale problems}

\section{Introduction and Motivations}\label{sec:intro}

The aim of the paper is to study a new notion of two-scale
convergence\footnote{A deep bibliographic research, shows that the idea here
presented is suggested in an paper (of the late seventies and so well before
the introduction of the notion of two-scale convergence) by
{\tmname{Papanicolau}} and {\tmname{Varadhan}} {\cite{papanicolaou97boundary}}
in the context of stochastic homogenization.} which is very natural and, in
our opinion, gives a more straightforward approach to the homogenization
process: while removing the bother of the admissibility of test functions
{\cite{Allaire1992,lukkassen2002two}}, it nevertheless simplifies the proof of
all standard compactness results which made {\tmem{classical}} two-scale
convergence (introduced in {\cite{Nguetseng1989,Allaire1992}}) very worthy of
interest.

Attempts to overcome the question of admissibility of test functions arising
in the definition of two-scale convergence have been the subject of various
authors
{\cite{cioranescu2002periodic,nechvatal2004alternative,valadier1997admissible}}.
Among them, the periodic unfolding method is considered one of the most
successful. The idea, as well as its nomenclature, is introduced
{\cite{cioranescu2002periodic}} where the authors exploit a natural, although
purely mathematical, intuition to recover two-scale convergence as a classical
functional weak convergence in a suitable larger space. This recovery process
is achieved by introducing the so-called unfolding operator which, roughly
speaking, turns a sequence of $1$-scale functions into a sequence of $2$-scale
functions.

On the other hand, as it is simple to show by playing with Lebesgue
differentiation theorem, the recovery process is not univocal, and many
alternatives are possible. In guessing the one presented below, we did not
rely on mathematical intuition only, but we found inspiration from the physics
of the homogenization process. That is why we think it is important to dwell
on some preliminary considerations before giving definitions, theorems and
proofs.

The paper is organized as follows: in Section \ref{sec:intro.newapproach} we
explain the idea behind the proposed approach which will be formalized in
Section \ref{sec:2NA}. In Section \ref{sec:3Compact} we establish compactness
results for the new notion of two-scale convergence which play a central role
in the homogenization process. In Section \ref{sec:4} we test the
effectiveness of our notion of convergence on the
{\guillemotleft}classical{\guillemotright} model problem in the theory of
homogenization, i.e the one associate to a family of linear second-order
elliptic partial differential equation with periodically oscillating
coefficients. Section \ref{sec:correctorH1conv} is devoted to the so-called
first-order corrector results which aim to improve the convergence of the
solution gradients by adding corrector term. In Section \ref{sec:6bl} we
introduce the well-known boundary layer terms which aim to compensate the fast
oscillation of the family of solutions near the boundary. Eventually, in
Section \ref{sec:homharmonicmaps} we test the approach on a nonlinear problem:
we prove a weak two-scale compactness result for $\SStwo$-valued stationary
harmonic map, and make some remarks which point out some possible weaknesses
of this alternative notion of two-scale convergence.

\section{The cell averaging approach to periodic
homogenization}\label{sec:intro.newapproach}

\subsection{The classical two-scale convergence approach to periodic
homogenization}\label{subsec:sec1classicaltwoscale}

Let us focus on the classical model problem in homogenization: a linear
second-order partial differential equation with periodically oscillating
coefficients. Such an equations models, for example, the stationary heat
conduction in a periodic composite medium {\cite{Allaire1992,Donato1999}}. We
denote by $\Omega$ the material domain (a bounded open set in $\RR^N$) and by
$\YY \assign [0, 1]^N$ the unit cell of $\RR^N$. Denoting by $f \in L^2
(\Omega)$ the source term and enforcing a Dirichlet boundary condition for the
unknown $u_{\varepsilon}$, the model equation reads as
\begin{equation}
  - \tmop{div} (A_{\varepsilon} \nabla u_{\varepsilon}) = f \quad \text{in }
  \Omega, \quad u_{\varepsilon} = 0 \quad \text{on } \partial \Omega,
  \label{eq:model-problem}
\end{equation}
where, for any $\varepsilon > 0$, we have defined $A_{\varepsilon}$ by
$A_{\varepsilon} (x) \assign A (x / \varepsilon)$, with $A$ (the so-called
{\tmem{matrix of diffusion coefficients}}) an $L^{\infty}$ and $\YY$-periodic
matrix valued function, which is uniformly coercive, i.e. such that for two
positive constants $0 < \alpha \leqslant \beta$ one has (for a.e. $y \in \YY$)
$\alpha | \xi |^2 \leqslant A (y) \xi \cdot \xi \leqslant \beta | \xi |^2$ for
every $\xi \in \RR^N$. Here we have supposed $A$ depending on the periodic
variable only although later we will work with the more general case in which
$A$ depends on the $x$ variable too. The weak formulation of problem
(\ref{eq:model-problem}) reads as:
\begin{equation}
  \int_{\Omega} A_{\varepsilon} \nabla u_{\varepsilon} \cdot \nabla \varphi =
  \int_{\Omega} f \varphi, \label{eq:model-problemwf}
\end{equation}
and according to Lax-Milgram theorem for each $\varepsilon > 0$ there exists a
unique weak solution $u_{\varepsilon} \in H_0^1 (\Omega)$ of
(\ref{eq:model-problemwf}). The family of solutions
$(u_{\varepsilon})_{_{\varepsilon \in \RR^+}}$ and the family of
{\tmem{fluxes}} $(\xi_{\varepsilon})_{\varepsilon \in \RR^+} \assign
(A_{\varepsilon} \nabla u_{\varepsilon})_{\varepsilon \in \RR^+}$, constitute
bounded subsets respectively of $H_0^1 (\Omega)$ and $L^2 (\Omega)$. Thus
there exist subfamilies (that we still denote by
$(u_{\varepsilon})_{\varepsilon \in \RR^+}$ and
$(\xi_{\varepsilon})_{\varepsilon \in \RR^+}$) and elements $u_0 \in H_0^1
(\Omega), \xi_0 \in L^2 (\Omega)$ such that $\nabla u_{\varepsilon}
\rightharpoonup \nabla u_0$ and $\xi_{\varepsilon} \rightharpoonup \xi_0$
weakly in $L^2 (\Omega)$. Hence, passing to the limit in
(\ref{eq:model-problemwf}), we get $(\xi_0, \nabla \varphi)_{L^2 (\Omega)} =
(f, \varphi)_{L^2 (\Omega)}$, where the limit flux $\xi_0$ is the weak limit
of the product of the weakly convergent sequences $\nabla u_{\varepsilon}
\rightharpoonup \nabla u_0$ and $A_{\varepsilon} \rightharpoonup \langle A
\rangle_{\YY}$. The identification of the limit flux $\xi_0$ in terms of $u_0$
and $A$ is the first aim in the mathematical theory of periodic
homogenization.

A procedure for the homogenization of problem (\ref{eq:model-problem})
appeared in 1989 by the means of the so-called two-scale convergence. This
notion, introduced for the first time by {\tmname{Nguetseng}} in
{\cite{Nguetseng1989}}, was later named {\guillemotleft}two-scale
convergence{\guillemotright} by {\tmname{Allaire}} {\cite{Allaire1992}} who
further developed the notion by giving more direct proofs of the main
compactness results. To better understand the idea behind the classical
two-scale approach, let us recall the following compactness results
{\cite{Allaire1992}}, from which the notion of two-scale convergence
originates:

\begin{proposition}[Nguetseng {\cite{Nguetseng1989}}, Allaire
{\cite{Allaire1992}}]
  If $(u_{\varepsilon})_{\varepsilon \in \RR^+}$ is a bounded sequence in $L^2
  (\Omega)$, there exists $u_0 \in L^2 \left( \Omega \times \YY \right)$, such
  that, up to a subsequence
  \begin{equation}
    \lim_{\varepsilon \rightarrow 0} \int_{\Omega} u_{\varepsilon} (x) \varphi
    (x, x / \varepsilon) \mathd x = \int_{\Omega \times \YY} u_0 (x, y)
    \varphi (x, y) \mathd x \; \mathd y
  \end{equation}
  for any test function\footnote{As it is classical in the field, we index by
  $\sharp$ spaces that consist of periodic functions.} $\varphi \in
  \mathcal{D} [ \Omega, C^{\infty}_{\sharp} \left( \YY \right) ]$.
  Moreover, if $(u_{\varepsilon})_{\varepsilon \in \RR^+}$ is a bounded
  sequence in $H^1 (\Omega)$, then there exist functions $u_0 \in H^1
  (\Omega)$ and $u_1 \in L^2 [ \Omega, H^1_{\sharp} \left( \YY \right) /
  \RR ]$ such that, up to a subsequence
  \begin{equation}
    \lim_{\varepsilon \rightarrow 0} \int_{\Omega} \nabla u_{\varepsilon} (x)
    \cdot \psi (x, x / \varepsilon) \mathd x = \int_{\Omega \times \YY} \left(
    \gradx u_0 (x) + \grady u_1 (x, y) \right) \cdot \psi (x, y) \mathd x \;
    \mathd y
  \end{equation}
  for any test function $\psi \in \mathcal{D} [ \Omega,
  C^{\infty}_{\sharp} \left( \YY \right) ]^N$.
\end{proposition}

\begin{figure}[t]
  \includegraphics{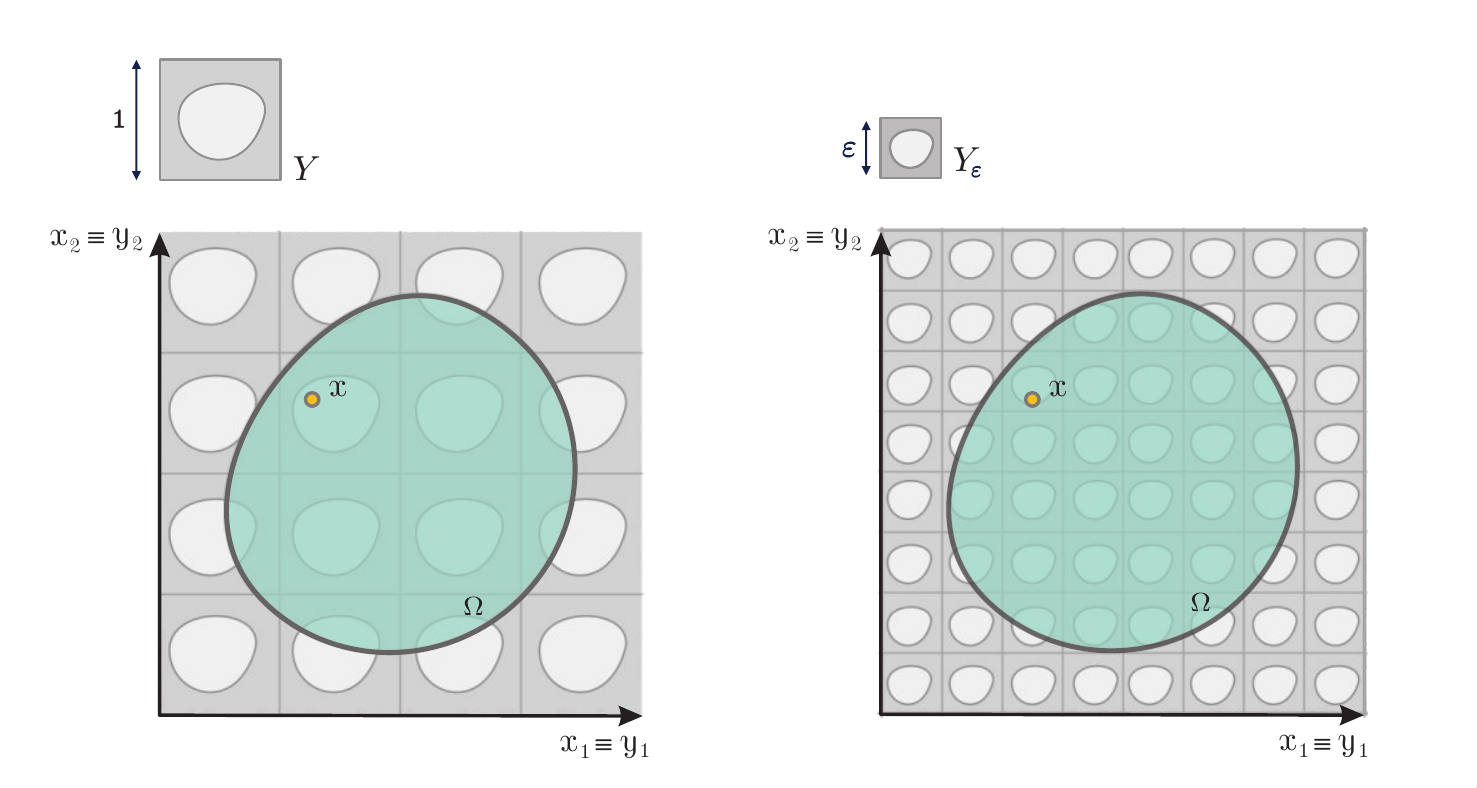}
  \caption{\label{fig:perhom1}If we assume that the heterogeneities are evenly
  distributed inside the media $\Omega$, we can model the material as
  periodic. As illustrated in the figure, this means that we can think of the
  material as being immersed in a grid of small identical cubes $\YYe$, the
  side-length of which is $\varepsilon$.}
\end{figure}

It is then natural to give the following (see {\cite{Allaire1992}})

\begin{definition}[Allaire {\cite{Allaire1992}}]
  A sequence of functions $u_{\varepsilon}$ in $L^2 (\Omega)$
  {\tmem{two-scale}} converges to a limit $u_0 \in L^2 \left( \Omega \times
  \YY \right)$ if, for any function $\varphi \in \mathcal{D} [ \Omega,
  C^{\infty}_{\sharp} \left( \YY \right) ]$ we have
  \begin{equation}
    \lim_{\varepsilon \rightarrow 0} \int_{\Omega} u_{\varepsilon} (x) \varphi
    (x, x / \varepsilon) \mathd x = \int_{\Omega \times \YY} u_0 (x, y)
    \varphi (x, y) \mathd x \; \mathd y .
  \end{equation}
  In that case we write $u_{\varepsilon} \twos u_0$. We say that the sequence
  $(u_{\varepsilon})$ strongly two-scale converges to a limit $u_0 \in L^2
  \left( \Omega \times \YY \right)$, if $u_{\varepsilon} \twos u_0$ and $\|
  u_0 \|_{L^2 \left( \Omega \times \YY \right)} = \lim_{\varepsilon
  \rightarrow 0} \| u_{\varepsilon} \|_{L^2 (\Omega)}$.
\end{definition}

It is now immediate to understand the role played by two-scale convergence in
the homogenization process. Indeed, by writing (\ref{eq:model-problemwf}) in
the form
\begin{equation}
  \int_{\Omega} \nabla u_{\varepsilon} (x) \cdot A^{\T} \left(
  \frac{x}{\varepsilon} \right) \nabla \varphi_{\varepsilon} (x) \mathd x \eqs
  \int_{\Omega} f (x) \varphi_{\varepsilon} (x) \mathd x,
\end{equation}
and choosing the right shape for the test functions $\varphi_{\varepsilon}$,
it is possible to interpret the left-hand side of the previous relation as the
product of a strongly two-scale convergent sequence (namely
$A^{\T}_{\varepsilon} \nabla \varphi_{\varepsilon} (x)$) with the weakly
two-scale convergent sequence $\nabla u_{\varepsilon}$, from which weak
two-scale convergence of the product, and hence the homogenized equation,
easily follows (cfr. {\cite{Allaire1992,Donato1999}} for details).

Unfortunately, for this procedure to be possible it is essential to add a
technical hypothesis: the sequence of coefficients $(A_{\varepsilon})$ must be
{\tmem{admissible}} in the sense that (cfr. {\cite{Allaire1992}})
\begin{equation}
  \lim_{\varepsilon \rightarrow 0} \| A_{\varepsilon} \|_{L^2 (\Omega)} \eqs
  \| A \|_{L^2 \left( \Omega \times \YY \right)} .
\end{equation}
It turns out that this is a subtle notion. Indeed, for a given function $\psi
\in L^2_{\sharp} [ \YY, L^2 (\Omega) ]$ there is no reasonable way
to give a meaning to the {\guillemotleft}trace{\guillemotright} function $x
\mapsto \psi (x, x / \varepsilon)$. The complete space of admissible functions
is not known much more precisely. Functions in $L^p [ \Omega, C_{\sharp}
\left( \YY \right) ]$ as well as $L^p_{\sharp} [ \YY, C (\Omega)
]$ are admissible, but it is unclear how much the regularity of $\psi$
can be weakened: we refer to {\cite{Allaire1992}} for an explicit construction
of a non admissible function which belongs to $C [ \Omega, L^1_{\sharp}
\left( \YY \right)]$.

\subsection{The cell averaging idea}\label{subsec:intro.newapproach}

The {\guillemotleft}classical{\guillemotright} approach to periodic
homogenization originates by the modeling assumption that since the
heterogeneities are evenly distributed inside the media $\Omega$, we can think
of the material as being immersed in a grid of small identical cubes
$\YY_{\varepsilon}$, the side-length of which is $\varepsilon$ (see Figure
\ref{fig:perhom1}). If we denote by $\Omega_a \assign \Omega + a$, with $a \in
\RR^N$, a translated copy of $\Omega$ such that $\Omega \cap \Omega_a \neq
\emptyset$, this modeling approach assumes that, at scale $\varepsilon$, the
contribution of the diffusion coefficients at any $x \in \Omega \cap
\Omega_a$, is given by $A (x / \varepsilon)$ both if we focus on the problem
$- \tmop{div} (A_{\varepsilon} \nabla u_{\varepsilon}) = f$ in $\Omega$ and on
the problem $(f_a \assign f (x - a))$ $- \tmop{div} (A_{\varepsilon} \nabla
u_{\varepsilon}) = f_a$ in $\Omega_a$. Although this assumption is
mathematically reasonable when $\varepsilon$ tends to be very small, it is
nevertheless the reason why the two-scale convergence produces
{\guillemotleft}two-variables{\guillemotright} functions starting from a
family of {\guillemotleft}one-variable{\guillemotright} functions.

\begin{figure}[t]
  \includegraphics{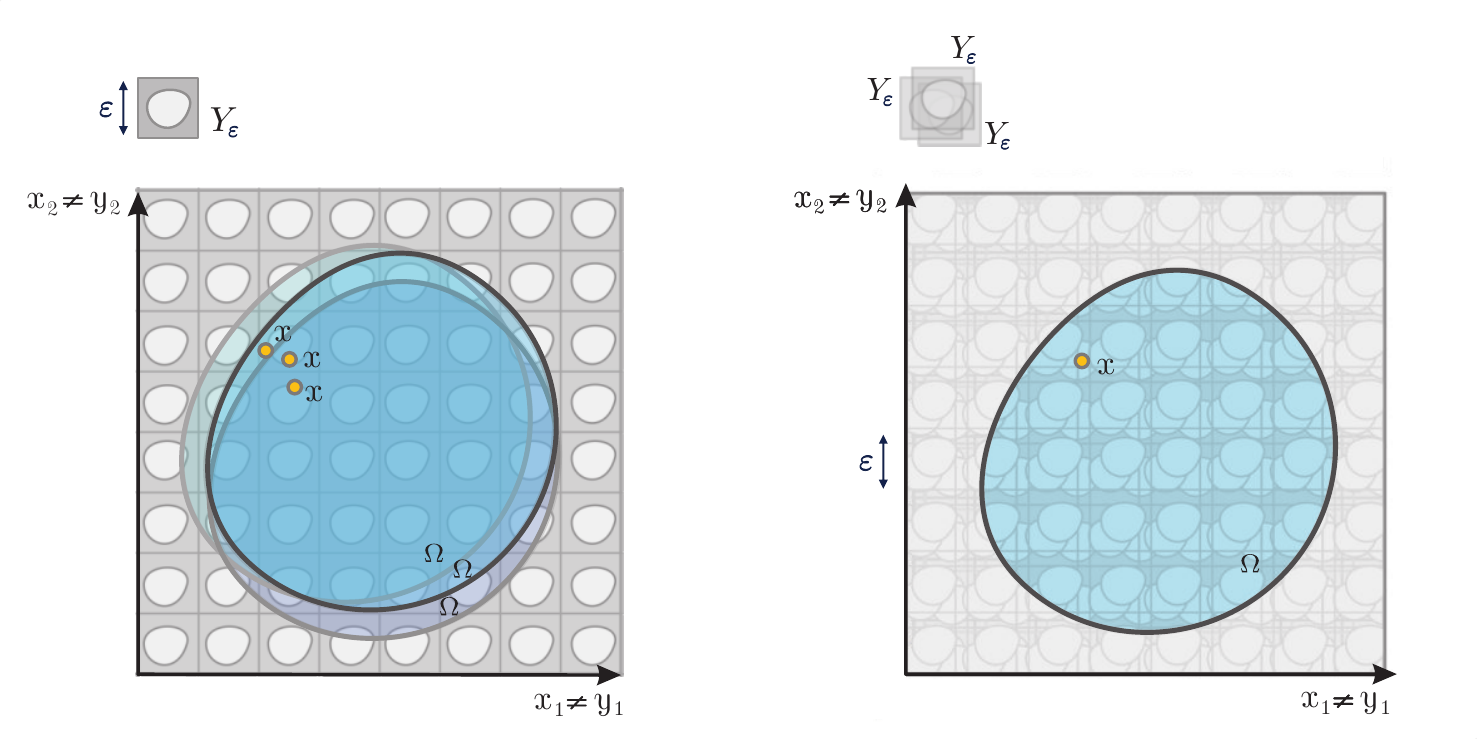}
  \caption{\label{fig:perhom2}More realistic is to think of the material as
  being immersed in a grid of small identical cubes $\YYe$, up to an unknown
  translation of size smaller than $\varepsilon$. We thus consider all
  possible translations, which we take into account by the introduction of a
  new variable.}
\end{figure}

On the other hand, it is clear that a more realistic approach consists in
taking into account the effects of the diffusion coefficients $A_{\varepsilon}
\assign A (x / \varepsilon)$ via a family of displacement of length at most
$\varepsilon$, i.e. via the family of diffusion coefficients $\left(
A_{\varepsilon} \left( \, \cdot \, + \varepsilon y \right)
\right)_{(\varepsilon, y) \in \RR^+ \times \YY} = (A (y + \cdot /
\varepsilon))_{(\varepsilon, y) \in \RR^+ \times \YY}$, and hence (see Figure
\ref{fig:perhom2}) via the family of boundary value problems depending on the
cell-size parameter $\varepsilon \in \RR^+$ and on the translation parameter
$y \in \YY$. The new homogenized problem then goes through the following two
steps: for every $\varepsilon \in \RR^+$ find (in a suitable sense) a
$\YY$-periodic solution $u_{\varepsilon} (x, y)$ of the Dirichlet problem
\begin{equation}
  - \tmop{div} (A_{\varepsilon} (x + \varepsilon y) \nabla u_{\varepsilon} (x,
  y)) = f (x) \quad \text{in } \Omega, \quad u_{\varepsilon} (x, y) = 0 \quad
  \text{on } \partial \Omega ; \label{eq:model-problemOmeganew}
\end{equation}
then take the average $\langle u_{\varepsilon} \rangle_{\YY}$ as a more
realistic modelization of the solution associated, at scale $\varepsilon$, to
evenly distributed heterogeneities inside the media $\Omega$.

In this framework the homogenization process demands for the computation of
the limiting behaviour, as $\varepsilon \rightarrow 0$, of the family of two
variable solutions $u_{\varepsilon} (x, y)$, i.e. for an asymptotic expansion
of the form
\begin{equation}
  u_{\varepsilon} (x, y) \eqs u_0 \left( x, y + \frac{x}{\varepsilon} \right)
  + \varepsilon \; u_1 \left( x, y + \frac{x}{\varepsilon} \right) +
  \varepsilon^2 \; u_2 \left( x, y + \frac{x}{\varepsilon} \right) + \cdots,
  \label{eq:asymptexpansion}
\end{equation}
in which $u_0$ is the solution of the homogenized equation and $u_1$ is the
so-called {\tmem{first order corrector}} (cfr. the analogues definitions in
{\cite{Allaire1992,Donato1999}}).

We are now in position to explain the new approach. To this end, let us
introduce the operator
\begin{equation}
  \mathcal{F}_{\varepsilon} : u \in L^2_{\sharp} [ \YY, L^2 (\Omega)
  ] \mapsto u (x, y - x / \varepsilon) \in L^2_{\sharp} [ \YY, L^2
  (\Omega) ] . \label{eq:cellshiftop1}
\end{equation}
Due to the $Y$-periodicity of $A$, the variational formulation of
(\ref{eq:model-problemOmeganew}) reads as the problem of finding
$u_{\varepsilon} \in L^2_{\sharp} [ \YY, H^1_0 (\Omega) ]$ such
that
\begin{equation}
  \int_{\Omega \times \YY} A (y) \mathcal{F}_{\varepsilon} \left( \gradx
  u_{\varepsilon} \right) (x, y) \cdot \mathcal{F}_{\varepsilon} \left( \gradx
  \psi_{\varepsilon} \right) (x, y) \mathd x \mathd y \eqs \int_{\Omega \times
  \YY} f (x) \mathcal{F}_{\varepsilon} (\psi_{\varepsilon}) (x, y) \mathd x
  \mathd y \label{eq:newweakform}
\end{equation}
for every $\psi_{\varepsilon} \in L^2_{\sharp} [ \YY, H^1_0 (\Omega)
]$. Therefore, if $\mathcal{F}_{\varepsilon} \left( \gradx
u_{\varepsilon} \right) \rightharpoonup \tmmathbf{v}$ weakly in $L^2_{\sharp}
[ \YY, L^2 (\Omega) ]$, then for every couple of
{\guillemotleft}test functions{\guillemotright} $\psi, \tmmathbf{\psi} \in
L^2_{\sharp} [ \YY, L^2 (\Omega) ]$ such that for some family
$\psi_{\varepsilon} \in L^2_{\sharp} [ \YY, H^1_0 (\Omega) ]$ we
have $\mathcal{F}_{\varepsilon} (\psi_{\varepsilon}) \rightarrow \psi$ and
$\mathcal{F}_{\varepsilon} \left( \gradx \psi_{\varepsilon} \right)
\rightarrow \tmmathbf{\psi}$ strongly in $L^2_{\sharp} [ \YY, L^2
(\Omega) ]$, passing to the limit in (\ref{eq:newweakform}), we finish
with the {\guillemotleft}homogenized equation{\guillemotright}
\begin{equation}
  \int_{\Omega \times \YY} A (y) \tmmathbf{v} (x, y) \cdot \tmmathbf{\psi} (x,
  y) \mathd x \mathd y \eqs \int_{\Omega \times \YY} f (x) \psi (x, y) \mathd
  x \mathd y.
\end{equation}
Of course, to find an explicit expression for the homogenized equation, and
more generally to build a kind of two-scale calculus, it is important to
investigate the interconnections between the convergence of the families
$u_{\varepsilon}$ and $\mathcal{F}_{\varepsilon} (u_{\varepsilon})$ in
$L^2_{\sharp} [ \YY, H^1 (\Omega) ]$, and to understand which are
the subspaces of $L^2_{\sharp} [ \YY, H^1 (\Omega) ]$ which are
reachable by strong convergence of family of the type
$\mathcal{F}_{\varepsilon} (\varphi_{\varepsilon})$ in $L^2_{\sharp} [
\YY, H^1 (\Omega) ]$. This and many other important aspects of the
question are the object of the next two sections.

\section{The alternative approach to two-scale convergence}\label{sec:2NA}

\subsection{Notation and preliminary definitions}

In what follows we denote by $\YY = [0, 1]^N$ the unit cell of $\RR^N$ and by
$\Omega$ an open set of $\RR^N$. For any measurable function $u$ defined on
$\YY$ we denote by $\langle u \rangle_{\YY}$ the integral average of $u$.

By $C^{\infty}_{\sharp} [ \YY, \mathcal{D} (\Omega) ]$ we mean the
vector space of test functions $u : \Omega \times \RR^N \rightarrow \RR$ such
that the section $u (x, \cdot) \in C^{\infty}_{\sharp} \left( \YY \right)$ for
every $x \in \Omega$, and the section $u (\cdot, y) \in \mathcal{D} (\Omega)$
for every $y \in \RR^N$. Similarly we denote by $L^2_{\sharp} [ \YY, L^2
(\Omega) ]$ the Hilbert space of $\YY$-periodic distributions which are
in $L^2 \left( \Omega \times \YY \right)$, and by $L^2_{\sharp} [ \YY,
H^1 (\Omega) ]$ the Hilbert subspace of $L^2_{\sharp} [ \YY, L^2
(\Omega) ]$ constituted of distributions $u$ such that $\gradx u \in
L^2_{\sharp} [ \YY, L^2 (\Omega) ]$.

Next, we denote by $L^2 [ \Omega ; H^1_{\sharp} \left( \YY \right)
]$ the Hilbert space of $\YY$-periodic distributions $u \in \mathcal{D}'
\left( \Omega \times \RR^N \right)$ such that $u (\cdot, y) \in L^2 (\Omega)$
for a.e. $y \in \YY$ and $u (x, \cdot) \in H^1_{\tmop{loc}} \left( \RR^N
\right)$ for a.e. $x \in \Omega$.

Finally, in the next Proposition \ref{prop:isometricFeps}, we denote by
$\mathcal{E}'_{\sharp} [ \YY, \mathcal{D}' (\Omega) ]$ the
algebraic dual of $C^{\infty}_{\sharp} [ \YY, \mathcal{D} (\Omega)
]$, and for any $u \in \mathcal{E}'_{\sharp} [ \YY, \mathcal{D}'
(\Omega) ]$ and any $\tmmathbf{\psi} \in C^{\infty}_{\sharp} [ \YY,
\mathcal{D} (\Omega) ]^N$ we define the partial gradient $\gradx u$ by
the position $\left\langle \gradx u, \tmmathbf{\psi} \right\rangle \assign -
\left\langle u, \divx \tmmathbf{\psi} \right\rangle$ and the
$\varepsilon$-cell shifting of $u$ by the position $\langle u (x, y - x /
\varepsilon), \tmmathbf{\psi} (x, y) \rangle \assign \langle u (x, y),
\tmmathbf{\psi} (x, y + x / \varepsilon) \rangle$.

\subsection{Cell averaging two-scale convergence}

Motivated by the considerations made in subsection
\ref{subsec:intro.newapproach} we give the following

\begin{definition}
  Let $\Omega \subseteq \RR^N$ be an open set and $\YY$ the unit cell of
  $\RR^N$. For any $\varepsilon > 0$, we define the
  $\varepsilon$-{\mybold{cell shift operator}} $\mathcal{F}_{\varepsilon}$ by
  the position
  \begin{equation}
    u \in L^2_{\sharp} [ \YY, L^2 (\Omega) ] \mapsto
    \mathcal{F}_{\varepsilon} (u) \assign u (x, y - x / \varepsilon) \in
    L^2_{\sharp} [ \YY, L^2 (\Omega) ],
  \end{equation}
  i.e. as the composition of $u$ with the diffeomorphism $(x, y) \in \Omega
  \times \RR^N \mapsto (x, y - x / \varepsilon) \in \Omega \times \RR^N$. We
  then denote by $\mathcal{F}^{\ast}_{\varepsilon}$ the algebraic adjoint
  operator which maps $u (x, y)$ to $u (x, y + x / \varepsilon)$.
\end{definition}

\begin{definition}
  A sequence of $L^2_{\sharp} [ \YY, L^2 (\Omega) ]$ functions
  $(u_{\varepsilon})_{\varepsilon \in \RR^+}$ is said to weakly two-scale
  converges to a function $u_0 \in L^2_{\sharp} [ \YY, L^2 (\Omega)
  ]$, if $\mathcal{F}_{\varepsilon} (u_{\varepsilon}) \rightharpoonup
  u_0$ weakly in $L^2_{\sharp} [ \YY, L^2 (\Omega) ]$, i.e if and
  only if
  \begin{equation}
    \lim_{\varepsilon \rightarrow 0^+} \int_{\Omega \times \YY}
    u_{\varepsilon} \left( x, y - \frac{x}{\varepsilon} \right) \psi (x, y)
    \mathd x \mathd y = \int_{\Omega \times \YY} u_0 (x, y) \psi (x, y) \mathd
    x \mathd y,
  \end{equation}
  for every $\psi \in L^2_{\#} [ \YY, L^2 (\Omega) ]$. In that case
  we write $u_{\varepsilon} \twoheadrightarrow u_0$ weakly in $L^2_{\sharp}
  [ \YY, L^2 (\Omega) ]$. We say that $u_{\varepsilon}
  \twoheadrightarrow u_0$ strongly in $L^2_{\sharp} [ \YY, L^2 (\Omega)
  ]$ if $\mathcal{F}_{\varepsilon} (u_{\varepsilon}) \rightarrow u_0$
  strongly in $L^2_{\sharp} [ \YY, L^2 (\Omega) ]$.
\end{definition}

\begin{remark}
  We have stated the definition in the framework of square summable functions.
  Nevertheless, almost all of what we say here and hereinafter easily extends,
  with obvious modifications, to the setting of $L^p$ spaces.
\end{remark}

\begin{remark}
  Since the notion of two-scale convergence relies on the classical notion of
  weak convergence in Banach space, we immediately get, among others,
  boundedness in norm of weakly two-scale convergent sequences. This aspect is
  not captured by the classical notion of two-scale convergence which, by
  testing convergence on functions in $\mathcal{D} [ \Omega,
  C^{\infty}_{\sharp} \left( \YY \right) ]$, i.e. having compact support
  in $\Omega$, may cause loss of information on any concentration of
  {\guillemotleft}mass{\guillemotright} near the boundary of the sequence
  $(u_{\varepsilon})_{\varepsilon \in \RR^+}$ (cfr.
  {\cite{lukkassen2002two}}).
\end{remark}

We now state some properties of the operator $\mathcal{F}_{\varepsilon}$,
which are simple consequence of the definitions, and will be used extensively
(and sometime tacitly) in the sequel:

\begin{proposition}
  \label{prop:isometricFeps}Let $\varepsilon > 0$. The operator
  $\mathcal{F}_{\varepsilon}$ is an isometric isomorphism of $L^2_{\sharp}
  [ \YY, L^2 (\Omega) ]$ and the following relations hold:
  \begin{itemizedot}
    \item If $\tmmathbf{\psi} \in C^{\infty}_{\sharp} [ \YY, \mathcal{D}
    (\Omega) ]^N$ then $\mathcal{F}_{\varepsilon} (\tmmathbf{\psi}) \in
    C^{\infty}_{\sharp} [ \YY, \mathcal{D} (\Omega) ]^N$ and one
    has
    \begin{equation}
      \divx \mathcal{F}_{\varepsilon} (\tmmathbf{\psi}) =
      \mathcal{F}_{\varepsilon} \left( \divx \tmmathbf{\psi} \right) -
      \frac{1}{\varepsilon} \mathcal{F}_{\varepsilon} \left( \divy
      \tmmathbf{\psi} \right), \quad \divy \mathcal{F}_{\varepsilon}
      (\tmmathbf{\psi}) =\mathcal{F}_{\varepsilon} \left( \divy
      \tmmathbf{\psi} \right) .  \label{eqpropdiv2}
    \end{equation}
  \end{itemizedot}
  Next, let us denote by $\mathcal{E}'_{\sharp} [ \YY, \mathcal{D}'
  (\Omega) ]$ the algebraic dual of $C^{\infty}_{\sharp} [ \YY,
  \mathcal{D} (\Omega) ]$:
  \begin{itemizedot}
    \item If $u \in \mathcal{E}'_{\sharp} [ \YY, \mathcal{D}' (\Omega)
    ]$ then $\mathcal{F}_{\varepsilon} (u) \in \mathcal{E}'_{\sharp}
    [ \YY, \mathcal{D}' (\Omega) ]$ and one has
    \begin{equation}
      \left\langle \gradx [\mathcal{F}_{\varepsilon} (u)], \tmmathbf{\psi}
      \right\rangle = \left\langle \mathcal{F}_{\varepsilon} \left( \gradx u
      \right) - \frac{1}{\varepsilon} \mathcal{F}_{\varepsilon} \left( \grady
      u \right), \tmmathbf{\psi} \right\rangle, \quad \left\langle
      \mathcal{F}_{\varepsilon} \left( \grady u \right), \tmmathbf{\psi}
      \right\rangle = \left\langle \grady [\mathcal{F}_{\varepsilon} (u)],
      \tmmathbf{\psi} \right\rangle,
      \label{eq:gradandFeps}
    \end{equation}
    for any $\tmmathbf{\psi} \in C^{\infty}_{\sharp} [ \YY, \mathcal{D}
    (\Omega) ]^N$.
  \end{itemizedot}
\end{proposition}

\begin{proof}
  For every $u \in L^2_{\#} [ \YY, L^2 (\Omega) ]$, by the
  translational invariance of the integral over $\YY$ with respect to the
  section $u (x, \cdot) \in L^2 \left( \YY \right)$, we get
  \begin{equation}
    \| \mathcal{F}_{\varepsilon} (u) \|_{L^2 \left( \Omega \times \YY \right)}
    \eqs \left( \int_{\Omega \times \YY} | u (x, y - x / \varepsilon) |^2
    \right)^{1 / 2} \eqs \| u \|_{L^2 \left( \Omega \times \YY \right)} .
  \end{equation}
  Relation (\ref{eqpropdiv2}) is a standard computation. Equation
  (\ref{eq:gradandFeps}) is a direct consequence of (\ref{eqpropdiv2}). Indeed
  for any $\tmmathbf{\psi} \in C^{\infty}_{\sharp} [ \YY, \mathcal{D}
  (\Omega) ]^N$ we have 
	\begin{align}
		\left\langle \gradx [\mathcal{F}_{\varepsilon}
  (u)], \tmmathbf{\psi} \right\rangle & \assign - \left\langle u, \;
  \mathcal{F}_{\varepsilon}^{\ast} \left( \divx \tmmathbf{\psi} \right)
  \right\rangle \\
	&= - \left\langle u, \; \divx \mathcal{F}_{\varepsilon}^{\ast}
  (\tmmathbf{\psi}) - \frac{1}{\varepsilon} \divy
  \mathcal{F}_{\varepsilon}^{\ast} (\tmmathbf{\psi}) \right\rangle,
	\end{align}
	and this
  last expression is nothing else than (\ref{eq:gradandFeps}).
\end{proof}

\section{Compactness results}\label{sec:3Compact}
As already pointed out, one of the greatest strengths of the new notion of
two-scale convergence is in the simplification we gain in proving compactness
results for that notion. In that regard it is important to remark that one of
the main contributions given by {\tmname{Allaire}} in {\cite{Allaire1992}} was
to give a concise proof of the nowadays classical compactness results
associated to two-scale convergence, by the means of Banach-Alaoglu theorem
and Riesz representation theorem for Radon measures (cfr. Theorem 1.2 in
{\cite{Allaire1992}}).

\subsection{Compactness in $L^2_{\sharp} [ \YY, L^2 (\Omega) ]$}
As as previously announced, the proof of the following compactness result is
completely straightforward (cfr. Theorem 1.2 in {\cite{Allaire1992}}).

\begin{theorem}
  \label{Thm:compL2}From every bounded subset $(u_{\varepsilon})_{\varepsilon
  > 0}$ of $L^2_{\sharp} [ \YY, L^2 (\Omega) ]$ is possible to
  extract a weakly two-scale convergent sequence.
\end{theorem}

\begin{proof}
  According to Proposition \ref{prop:isometricFeps},
  $\mathcal{F}_{\varepsilon}$ is an isometric isomorphism of $L^2_{\sharp}
  [ \YY, L^2 (\Omega) ]$ in it, and therefore also
  $(\mathcal{F}_{\varepsilon} (u_{\varepsilon}))_{\varepsilon > 0}$ is a
  bounded subset of $L^2_{\sharp} [ \YY, L^2 (\Omega) ]$. Therefore
  there exists an $u_0 \in L^2_{\sharp} [ \YY, L^2 (\Omega) ]$ and
  a subsequence extracted from $(u_{\varepsilon})$, still denoted by
  $(u_{\varepsilon})$, such that $\mathcal{F}_{\varepsilon} (u_{\varepsilon})
  \rightharpoonup u_0$ in $L^2_{\sharp} [ \YY, L^2 (\Omega) ]$,
  i.e. such that $u_{\varepsilon} \twoheadrightarrow u_0$ in $L^2_{\sharp}
  [ \YY, L^2 (\Omega) ]$.
\end{proof}

\subsection{Compactness in $L^2_{\sharp} [ \YY, H^1 (\Omega) ]$}

The following compactness results are the counterparts of the well-known
corresponding results for the classical notion two-scale convergence (cfr.
Proposition 1.14 in {\cite{Allaire1992}}).

\begin{proposition}
  \label{prop:g1classicalctwoscaleL2H1}Let $(u_{\varepsilon})$ be a sequence
  in $L^2_{\sharp} [ \YY, H^1 (\Omega) ]$ such that for some $(u_0,
  \tmmathbf{v}) \in L^2_{\sharp} [ \YY, L^2 (\Omega) ]^{N + 1}$ one
  has
  \begin{equation}
    u_{\varepsilon} \twoheadrightarrow u_0 \quad \text{in } L^2_{\sharp}
    [ \YY, L^2 (\Omega) ] \quad, \quad \gradx u_{\varepsilon}
    \twoheadrightarrow \tmmathbf{v} \quad \text{in } L^2_{\sharp} [ \YY,
    L^2 (\Omega) ]^N,
  \end{equation}
  then $u_0 (x, y) = \langle u_0 (x, \cdot) \rangle_{\YY}$, i.e. the two-scale
  limit $u_0$ does not depends on the $y$ variable. Moreover there exists an
  element $u_1 \in L^2 [ \Omega ; H^1_{\sharp} \left( \YY \right)
  ]$ such that $\tmmathbf{v}= \gradx u_0 + \grady u_1$.
\end{proposition}

\begin{proof}
  The relation $\gradx u_{\varepsilon} \twoheadrightarrow \tmmathbf{v}$ in
  $L^2_{\sharp} [ \YY, L^2 (\Omega) ]^N$ means, in particular, that
  for $\varepsilon \rightarrow 0$ one has $\left\langle
  \mathcal{F}_{\varepsilon} \left( \gradx u_{\varepsilon} \right),
  \tmmathbf{\psi} \right\rangle \rightarrow \langle \tmmathbf{v},
  \tmmathbf{\psi} \rangle$ for any $\tmmathbf{\psi} \in C^{\infty}_{\sharp}
  [ \YY, \mathcal{D} (\Omega) ]^N$. Moreover, from
  (\ref{eqpropdiv2}) we get
  \begin{eqnarray}
    \int_{\Omega \times \YY} \gradx [\mathcal{F}_{\varepsilon} (u)] (x, y)
    \cdot \tmmathbf{\psi} (x, y) \mathd x \mathd y & \eqs & - \int_{\Omega
    \times \YY} \mathcal{F}_{\varepsilon} (u_{\varepsilon}) (x, y) \divx
    \tmmathbf{\psi} (x, y) \mathd x \mathd y \nonumber\\
    &  & \qquad - \frac{1}{\varepsilon} \int_{\Omega \times \YY}
    \mathcal{F}_{\varepsilon} (u_{\varepsilon}) (x, y)  \divy \tmmathbf{\psi}
    (x, y) \mathd x \mathd y  \label{eq:templemmaindipendencefromy0L2H10}\\
    & \eqs & \int_{\Omega \times \YY} \mathcal{F}_{\varepsilon} \left( \gradx
    u \right) (x, y) \cdot \tmmathbf{\psi} (x, y) \mathd x \mathd y
    \nonumber\\
    &  & \qquad - \frac{1}{\varepsilon} \int_{\Omega \times \YY}
    \mathcal{F}_{\varepsilon} \left( \grady u \right) (x, y) \cdot
    \tmmathbf{\psi} (x, y) \mathd x \mathd y, 
    \label{eq:templemmaindipendencefromy0L2H1}
  \end{eqnarray}
  for any $\tmmathbf{\psi} \in C^{\infty}_{\sharp} [ \YY, \mathcal{D}
  (\Omega) ]^N$. Let us investigate the implications of
  (\ref{eq:templemmaindipendencefromy0L2H10}) and
  (\ref{eq:templemmaindipendencefromy0L2H1}). Since $\mathcal{F}_{\varepsilon}
  (u_{\varepsilon}) \rightharpoonup u_0$ and $\mathcal{F}_{\varepsilon} \left(
  \gradx u_{\varepsilon} \right) \rightharpoonup \tmmathbf{v}$, multiplying
  both members of relation (\ref{eq:templemmaindipendencefromy0L2H10}) by
  $\varepsilon$ and then letting $\varepsilon \rightarrow 0$ we get
  \begin{equation}
    \int_{\Omega \times \YY} u_0 (x, y)  \divy \tmmathbf{\psi} (x, y) \mathd x
    \mathd y \eqs 0 \quad \forall \tmmathbf{\psi} \in C^{\infty}_{\sharp}
    [ \YY, \mathcal{D} (\Omega) ]^N,
  \end{equation}
  from which the independence of the two-scale limit $u_0$ from the $y$
  variable follows. Thus for the limit function we have $u_0 (x, y) = \langle
  u_0 (x, \cdot) \rangle_{\YY}$ for every $y \in \YY$.
  
  On the other hand, from (\ref{eq:templemmaindipendencefromy0L2H1}), for
  every $\tmmathbf{\psi} \in C^{\infty}_{\sharp} [ \YY, \mathcal{D}
  (\Omega) ]^N$ such that $\divy \tmmathbf{\psi}= 0$ we have
  \begin{equation}
    \int_{\Omega \times \YY} \left( \mathcal{F}_{\varepsilon} \left( \gradx
    u_{\varepsilon} \right) (x, y) - \gradx [\mathcal{F}_{\varepsilon}
    (u_{\varepsilon})] (x, y)_{_{_{_{_{}}}}} \right) \cdot \tmmathbf{\psi} (x,
    y) \mathd x \mathd y \eqs 0.
  \end{equation}
  Since $\mathcal{F}_{\varepsilon} (u_{\varepsilon}) \rightharpoonup u_0$ in
  $L^2_{\sharp} [ \YY, L^2 (\Omega) ]$ one has $\gradx
  [\mathcal{F}_{\varepsilon} (u_{\varepsilon})] \rightarrow \gradx u_0$ in the
  sense of distribution; thus multiplying both members of the previous
  relation by $\varepsilon$ and then letting $\varepsilon \rightarrow 0$ we
  get (by hypothesis $\gradx u_{\varepsilon} \twoheadrightarrow \tmmathbf{v}$)
  \begin{equation}
    \int_{\Omega \times \YY} \left( \tmmathbf{v} (x, y) - \gradx u_0 (x, y)
    \right) \cdot \tmmathbf{\psi} (x, y) \mathd x \mathd y \eqs 0,
  \end{equation}
  for every $\tmmathbf{\psi} \in C^{\infty}_{\sharp} [ \YY, \mathcal{D}
  (\Omega) ]^N$ such that $\divy \tmmathbf{\psi}= 0$. According to
  De\,Rham's theorem, which in our context can be easily proved by means of
  Fourier series on $\YY$ (see e.g.$\;${\cite{Jikov1994}} p.6), the orthogonal
  complement of divergence-free functions are exactly the gradients, and
  therefore there exists a $u_1 \in L^2 [ \Omega ; H^1_{\sharp} \left(
  \YY \right) ]$ such that $\grady u_1 =\tmmathbf{v}- \gradx u_0$. This
  concludes the proof.
\end{proof}

\begin{proposition}
  \label{prop:g1classicalctwoscaleL2H1eps}Let $(u_{\varepsilon})$ be a
  sequence in $L^2_{\sharp} [ \YY, H^1 (\Omega) ]$ such that for
  some $(u_0, \tmmathbf{v}) \in [ L^2 \left( \Omega \times \YY \right)
  ]^{N + 1}$ one has
  \begin{eqnarray}
    u_{\varepsilon} \twoheadrightarrow u_0 \quad \text{in } L^2_{\sharp}
    [ \YY, L^2 (\Omega) ] & \lo{\tmop{and}} & \varepsilon \gradx
    u_{\varepsilon} \twoheadrightarrow \tmmathbf{v} \quad \text{in }
    L^2_{\sharp} [ \YY, L^2 (\Omega) ]^N, 
  \end{eqnarray}
  then $\tmmathbf{v}= \grady u_0$.
\end{proposition}

\begin{proof}
  As in the proof of Proposition \ref{prop:g1classicalctwoscaleL2H1} we have:
  \begin{eqnarray}
    \int_{\Omega \times \YY} \mathcal{F}_{\varepsilon} \left( \varepsilon
    \gradx u_{\varepsilon} \right) (x, y) \cdot \tmmathbf{\psi} (x, y) \mathd
    x \mathd y & \eqs & - \varepsilon \int_{\Omega \times \YY}
    \mathcal{F}_{\varepsilon} (u_{\varepsilon}) (x, y)  \divx \tmmathbf{\psi}
    (x, y) \mathd x \mathd y \nonumber\\
    &  & \qquad - \int_{\Omega \times \YY} \mathcal{F}_{\varepsilon}
    (u_{\varepsilon}) (x, y)  \divy \tmmathbf{\psi} (x, y) \mathd x \mathd y. 
    \label{eq:templemmaindipendencefromy0L2H10eps}
  \end{eqnarray}
	Let us investigate the implications of
  (\ref{eq:templemmaindipendencefromy0L2H10eps}). Since
  $\mathcal{F}_{\varepsilon} \left( \varepsilon \gradx u_{\varepsilon} \right)
  \rightharpoonup \tmmathbf{v}$ in $[ L^2 \left( \Omega \times \YY
  \right) ]^N$ one has that $\gradx [\mathcal{F}_{\varepsilon}
  (u_{\varepsilon})] \rightharpoonup \gradx u_0$ in $[ \mathcal{D}'
  \left( \Omega \times \YY \right) ]^N$. Then taking the limit for
  $\varepsilon \rightarrow 0$ in relation
  (\ref{eq:templemmaindipendencefromy0L2H10eps}) and integrating by parts, we
  get $\left\langle \tmmathbf{v}- \grady u_0, \tmmathbf{\psi} \right\rangle =
  0$ in $\mathcal{D}' \left( \Omega \times \YY \right)$ and therefore
  $\tmmathbf{v}= \grady u_0$.
\end{proof}

\subsection{Test functions reachable by strong two-scale convergence}
As pointed out at the end of subsection \ref{subsec:intro.newapproach}, in
order to identify the system of homogenized equations it is important to
understand the subspaces of $L^2_{\sharp} [ \YY, H^1 (\Omega) ]$
which are reachable by strong convergence in $L^2_{\sharp} [ \YY, H^1
(\Omega) ]$ (cfr. Lemma 1.13 in {\cite{Allaire1992}}). Although this
question become a simple observation in our framework, we will make constantly
use of the following result which therefore state as a proposition in order to
reference it when used.

\begin{proposition}
  \label{prop:attainedtest}The following statements hold:
  \begin{enumeratenumeric}
    \item For every $\varphi \in \mathcal{D} (\Omega)$ there exists a sequence
    of functions $(\varphi_{\varepsilon})_{\varepsilon > 0}$ of $L^2_{\sharp}
    [ \YY, H^1 (\Omega) ]$ such that $\mathcal{F}_{\varepsilon}
    (\varphi_{\varepsilon}) = \varphi$ and $\mathcal{F}_{\varepsilon} \left(
    \gradx \varphi_{\varepsilon} \right) = \gradx \varphi$ for every
    $\varepsilon > 0$, so that obviously $\varphi_{\varepsilon}
    \twoheadrightarrow \varphi$ strongly $L^2_{\sharp} [ \YY, L^2
    (\Omega) ]$ and $\gradx \varphi_{\varepsilon} \twoheadrightarrow
    \gradx \varphi$ strongly in $L^2_{\sharp} [ \YY, L^2 (\Omega)
    ]^N$.
    
    \item Similarly, for every $\psi \in \mathcal{D} \left( \Omega \times \YY
    \right)$ there exists a sequence of functions $(\psi_{\varepsilon}) \in
    L^2_{\sharp} [ \YY, H^1 (\Omega) ]$ such that
    $\mathcal{F}_{\varepsilon} (\psi_{\varepsilon}) = \psi$ and $\varepsilon
    \mathcal{F}_{\varepsilon} \left( \gradx \psi_{\varepsilon} \right) =
    \grady \psi + \varepsilon \gradx \psi$ for every $\varepsilon > 0$. In
    particular $\psi_{\varepsilon} \twoheadrightarrow \psi$ strongly in
    $L^2_{\sharp} [ \YY, L^2 (\Omega) ]$ and $\varepsilon \gradx
    \psi_{\varepsilon} \twoheadrightarrow \grady \psi$ strongly in
    $L^2_{\sharp} [ \YY, L^2 (\Omega) ]^N$.
  \end{enumeratenumeric}
\end{proposition}

\begin{proof}
  For every $\varphi \in \mathcal{D} (\Omega)$ the constant family of
  functions defined by the position $\varphi_{\varepsilon} (x, y) \assign
  \varphi (x) \otimes 1 (y)$ is in $L^2_{\sharp} [ \YY, L^2 (\Omega)
  ]$, and is such that $\mathcal{F}_{\varepsilon}
  (\varphi_{\varepsilon}) = \varphi$. Therefore $\mathcal{F}_{\varepsilon}
  (\varphi_{\varepsilon})$ strongly converges to $\varphi$ in $L^2 \left(
  \Omega \times \YY \right)$ \ and \ $\mathcal{F}_{\varepsilon} \left( \gradx
  \varphi_{\varepsilon} \right) = \gradx \varphi$ strongly converges to
  $\gradx \varphi$ in $L^2_{\sharp} [ \YY, L^2 (\Omega) ]^N$.
  
  For the second part of the statement we note that for every $\psi \in
  \mathcal{D} \left( \Omega \times \YY \right)$ the family $\psi_{\varepsilon}
  (x, y) \assign \psi (x, y + x / \varepsilon)$ is in $L^2_{\sharp} [
  \YY, L^2 (\Omega) ]$, and is such that $\mathcal{F}_{\varepsilon}
  (\psi_{\varepsilon}) = \psi$. Hence $\psi_{\varepsilon} \twoheadrightarrow
  \psi_0$ strongly in $L^2_{\sharp} [ \YY, L^2 (\Omega) ]$.
  Moreover $\varepsilon \mathcal{F}_{\varepsilon} \left( \gradx
  \psi_{\varepsilon} \right) = \grady [\mathcal{F}_{\varepsilon}
  (\psi_{\varepsilon})] + \varepsilon \gradx [\mathcal{F}_{\varepsilon}
  (\psi_{\varepsilon})] = \grady \psi + \varepsilon \gradx \psi$ so that
  $\varepsilon \gradx \psi_{\varepsilon} \twoheadrightarrow \grady \psi$
  strongly in $L^2_{\sharp} [ \YY, L^2 (\Omega) ]$.
\end{proof}

\section{The {\guillemotleft}classical{\guillemotright} homogenization
problem}\label{sec:4}

In the mathematical literature, the elliptic equation introduced in subsection
\ref{subsec:sec1classicaltwoscale}, Eq.$\;$(\ref{eq:model-problem}), it is
nowadays simply referred to as the classical homogenization problem. This
classical problem has achieved the role of {\guillemotleft}benchmark
problem{\guillemotright} for new methods in periodic homogenization: Whenever
a new method for periodic homogenization emerges, it is customary to test it
by the ease it allows to solve the classical homogenization problem. This is
exactly the aim of this section. Of course, as pointed out in subsection
\ref{subsec:intro.newapproach}, our testing problem is slightly different as
the matrix of diffusion coefficients is now a function depending on a
parameter. Nevertheless, and this is a really important point, the homogenized
equations we get are exactly the ones arising from the homogenization of the
classical homogenization problem.

\subsection{The {\guillemotleft}classical{\guillemotright} homogenization
problem}

Let $\Omega$ be a bounded open set of $\RR^N$. Let $f$ be a given function in
$L^2 (\Omega)$. For every $y \in \YY$ we consider the following linear
second-order elliptic equation
\begin{eqnarray}
  - \divx [ A (x, y + x / \varepsilon) \gradx u_{\varepsilon} (x, y)
  ] & = \hspace{1.2em} f (x) & \hspace{1.2em} \text{in } \Omega 
  \label{eq:homogproblem2dot0feq}\\
  u_{\varepsilon} (x, y) & = \hspace{1.2em} 0 & \hspace{1.2em} \text{on }
  \partial \Omega,  \label{eq:homogproblem2dot0}
\end{eqnarray}
where $A \in [ L^{\infty} \left( \Omega \times \RR^N \right)
]^{N^2}$ is a (not necessarily symmetric) matrix valued function defined
on $\Omega \times \YY$ and $\YY$-periodic in the second variable. We also
suppose $A$ to be {\mybold{uniformly elliptic}}, i.e. there exists a positive
constants $\alpha > 0$ such that $\alpha | \xi |^2 \leqslant A (x, y) \xi
\cdot \xi$ for any $\xi \in \RR^N$ and every $(x, y) \in \Omega \times \YY$.

Following {\cite{Allaire1992}} we give the following

\begin{definition}
  \label{def:classhomandcorr}The homogenized equation is defined as
  \begin{eqnarray}
    - \tmop{div} [A_{\hom} (x) \nabla u_0 (x)] & = \; \; f (x) & \text{ in }
    \Omega  \label{eq:homogenizedpdej1}\\
    u (x) & = \hspace{1.2em} 0 & \text{ on } \partial \Omega 
  \end{eqnarray}
  where the matrix $A_{\hom}$ is given by
  \begin{equation}
    A_{\hom} = \left\langle A (x, \cdot) \left( I_N + \grady \tmmathbf{\chi}
    (x, \cdot) \right) \right\rangle_{\YY}, \label{eq:Ahom0}
  \end{equation}
  where $\tmmathbf{\chi} \assign (\chi_1, \chi_2, \ldots, \chi_N)$ is the
  so-called vector of correctors where for every $i \in \NN_N$ the function
  $\chi_i$ is the unique solution in the space $L^{\infty} [ \Omega,
  H^1_{\sharp} \left( \YY \right) / \RR ]$ of the cell problem:
  \begin{equation}
    - \divy [ A (x, y) \left( \grady \chi_i (x, y) + e_i \right) ]
    \eqs 0. \label{eq:splittedsoldistribcell}
  \end{equation}
\end{definition}

We then have

\begin{theorem}
  \label{thm:newhomapproach}For every $\varepsilon \in \RR^+$ there exists a
  unique solution $u_{\varepsilon} \in L^2_{\sharp} [ \YY, H^1_0 (\Omega)
  ]$ of the problem
  (\ref{eq:homogproblem2dot0feq})-(\ref{eq:homogproblem2dot0}).
  \begin{enumerate}
    \item The sequence $(u_{\varepsilon})_{\varepsilon \in \RR^+}$ of
    $L^2_{\sharp} [ \YY, H^1_0 (\Omega) ]$ solutions is such that
    \begin{equation}
      u_{\varepsilon} \twoheadrightarrow u_0 \quad, \quad \gradx
      u_{\varepsilon} \twoheadrightarrow \gradx u_0 + \grady u_1 \quad
      \text{in } L^2_{\sharp} [ \YY, L^2 (\Omega) ]
    \end{equation}
    where $(u_0, u_1)$ is the unique solution in $H_0^1 (\Omega) \times L^2
    [ \Omega, H^1_{\sharp} (Y) / \RR ]$ of the following two-scale
    homogenized system:
    \begin{eqnarray}
      - \divy [ A (x, y) \left( \gradx u_0 (x) + \grady u_1 (x, y)
      \right) ] & = \; \; 0 & \text{in } \Omega \times \YY, 
      \label{eq:homeq1distrib}\\
      - \divx [ \int_{\YY} A (x, y) \left( \gradx u_0 (x) + \grady u_1
      (x, y) \right) \mathd y ] & = \hspace{1.2em} f (x) & \text{in }
      \Omega \times \YY .  \label{eq:homeq2}
    \end{eqnarray}
    \item Furthermore, the previous system in equivalent to the classical
    homogenized and cell equations through the relation
    \begin{equation}
      u_1 (x, y) = \nabla u_0 (x) \cdot \tmmathbf{\chi} (x, y) .
      \label{eq:linkclassicaland2scale}
    \end{equation}
  \end{enumerate}
\end{theorem}

\begin{proof}
  {\tmstrong{1)}} We write the weak formulation of problem
  (\ref{eq:homogproblem2dot0feq})-(\ref{eq:homogproblem2dot0}) on the space
  $L^2_{\sharp} [ \YY ; H^1_0 (\Omega) ]$:
  \begin{equation}
    \int_{\YY \times \Omega} A (x, x / \varepsilon + y) \gradx u_{\varepsilon}
    (x, y) \cdot \gradx \psi_{\varepsilon} (x, y) \mathd x \mathd y \eqs
    \int_{\YY \times \Omega} f (x) \psi_{\varepsilon} (x, y) \mathd x \mathd
    z, \label{eq:twovarweakform}
  \end{equation}
  with $\psi_{\varepsilon} \in L^2_{\sharp} [ \YY, H^1_0 (\Omega)
  ]$. Once endowed the space $L^2_{\sharp} [ \YY, H^1_0 (\Omega)
  ]$ with the equivalent norm $u \in L^2_{\sharp} [ \YY, H^1_0
  (\Omega) ] \mapsto \left\| \gradx u \right\|^2_{\Omega \times \YY}$,
  due to Lax-Milgram theorem, for every $\varepsilon > 0$ there exists a
  unique solution $u_{\varepsilon} \in L^2_{\sharp} [ \YY, H^1_0 (\Omega)
  ]$ and moreover
  \begin{equation}
    \left\| \gradx u_{\varepsilon} \right\|_{L^2 \left( \Omega \times \YY
    \right)} \leqslant \frac{c_{\Omega}}{\alpha}  \| f \|_{L^2 (\Omega)}
    \label{eq:stabilityestimateLM}
  \end{equation}
  where we have denote by $c_{\Omega}$ the Poincar{\'e} constant for the space
  $H_0^1 (\Omega)$. As a consequence of the uniform bound (with respect to
  $\varepsilon$) expressed by (\ref{eq:stabilityestimateLM}), taking into
  thanks to the reflexivity of the space $L^2_{\sharp} [ \YY, H^1_0
  (\Omega) ]$ and Proposition \ref{prop:g1classicalctwoscaleL2H1}, there
  exists a subsequence extracted from $(u_{\varepsilon})_{\varepsilon \in
  \RR^+}$, and still denoted by $(u_{\varepsilon})_{\varepsilon \in \RR^+}$,
  such that
  \begin{equation}
    u_{\varepsilon} \twoheadrightarrow u_0 \quad, \quad \gradx u_{\varepsilon}
    \twoheadrightarrow \gradx u_0 + \grady u_1 \quad \text{in } L^2_{\sharp}
    [ \YY, L^2 (\Omega) ],
  \end{equation}
  for a suitable $u_0 \in H^1_0 (\Omega)$ and $u_1 \in L^2 [\Omega,
  H^1_{\sharp} (Y)]$.
  
  Next we note that in terms of the operator $\mathcal{F}_{\varepsilon}$, the
  previous equation (\ref{eq:twovarweakform}) reads as
  \begin{eqnarray}
    \int_{\YY \times \Omega} \mathcal{F}_{\varepsilon} \left( \gradx
    u_{\varepsilon} \right) (x, y) \cdot A^{\T} (x, y)
    \mathcal{F}_{\varepsilon} \left( \gradx \psi_{\varepsilon} \right) (x, y)
    \mathd x \mathd y & \eq & \int_{\YY \times \Omega} f (x)
    \psi_{\varepsilon} (x, y) \mathd x \mathd y \nonumber\\
    & \eq & \int_{\YY \times \Omega} f (x) \mathcal{F}_{\varepsilon}
    (\psi_{\varepsilon}) (x, y) \mathd x \mathd y . 
    \label{eq:twovarweakformFeps}
  \end{eqnarray}
  Now, we already know that $\mathcal{F}_{\varepsilon} \left( \gradx
  u_{\varepsilon} \right) \rightharpoonup \gradx u_0 + \grady u_1$ in
  $L^2_{\sharp} [ \YY, L^2 (\Omega) ]$. We then observe that (cfr.
  Proposition \ref{prop:attainedtest}) for every $\varphi \in \mathcal{D}
  (\Omega)$, there exists a sequence $\psi_{\varepsilon}$ of $L^2_{\sharp}
  [ \YY, L^2 (\Omega) ]$ functions such that
  $\mathcal{F}_{\varepsilon} (\psi_{\varepsilon}) \rightarrow \varphi$ and
  $\mathcal{F}_{\varepsilon} \left( \gradx \psi_{\varepsilon} \right)
  \rightarrow \gradx \varphi$ strongly in $L^2_{\sharp} [ \YY, L^2
  (\Omega) ]$. Therefore passing to the limit for $\varepsilon
  \rightarrow 0$ in equation (\ref{eq:twovarweakformFeps}), we get
  \begin{equation}
    \int_{\YY \times \Omega} A (x, y) \left( \gradx u_0 (x) + \grady u_1 (x,
    y)_{_{_{_{_{}}}}} \right) \cdot \gradx \varphi (x) \mathd x \mathd y \eqs
    \int_{\Omega} f (x) \varphi (x) \mathd x, \label{eq:weakformtemp1}
  \end{equation}
  which, due to the arbitrariness of $\varphi \in \mathcal{D} (\Omega)$, in
  distributional form reads as (\ref{eq:homeq2}).
  
  On the other hand, for every $\psi \in \mathcal{D} \left( \Omega \times \YY
  \right)$ there exists (cfr. Proposition \ \ref{prop:attainedtest}) a family
  $(\psi_{\varepsilon})$ of $L^2_{\sharp} [ \YY, L^2 (\Omega) ]$
  functions such that $\varepsilon \mathcal{F}_{\varepsilon} \left( \gradx
  \psi_{\varepsilon} \right) \rightarrow \grady \psi$ strongly in $L^2 \left(
  \Omega \times \YY \right)$ so that, multiplying both members of
  (\ref{eq:twovarweakformFeps}) for $\varepsilon > 0$ and passing to the limit
  for $\varepsilon \rightarrow 0$ we get
  \begin{equation}
    \int_{\YY \times \Omega} A (x, y) \left( \gradx u_0 (x) + \grady u_1 (x,
    y) \right) \cdot \grady \psi (x, y) \mathd x \mathd y = 0
    \label{eq:homeq1}
  \end{equation}
  which, due to the arbitrariness of $\psi \in \mathcal{D} (\Omega)$, in
  distributional form reads as (\ref{eq:homeq1distrib}).
  
  We have thus proved that from any extracted subsequence from
  $(u_{\varepsilon})_{\varepsilon \in \RR^+}$ it is possible to extract a
  further subsequence which two-scale convergence to the solution of the
  system of equations (\ref{eq:weakformtemp1}),(\ref{eq:homeq1}). Since the
  system of equations (\ref{eq:weakformtemp1}),(\ref{eq:homeq1}) has only one
  solution $(u_0, u_1) \in H_0^1 (\Omega) \times L^2 [ \Omega,
  H^1_{\sharp} (Y) / \RR ]$, as it is immediate to check via Lax-Milgram
  theorem, the entire sequence $(u_{\varepsilon})_{\varepsilon \in \RR^+}$
  two-scale convergence to $u_0$.
\end{proof}

\begin{proof}
  {\tmstrong{2)}} The homogenization process has led to two partial
  differential equations, namely (\ref{eq:homeq1distrib}) and
  (\ref{eq:homeq2}). Let us observe that the distributional equation
  (\ref{eq:homeq1distrib}) can be equivalently written as
  \begin{equation}
    - \divy [ A (x, y) \left. \grady u_1 (x, y) \right) ] \eqs
    \divy A (x, y) \cdot \nabla u_0 (x), \label{eq:cellvect}
  \end{equation}
  where we have denoted by $\divy A = \left( \divy A_1, \; \divy A_2, \ldots,
  \, \divy A_N \right)$ the vector whose components are the $\divy$of the
  columns $A_1, A_2, \ldots, A_N$ of $A$. It is completely standard (see
  {\cite{papanicolau1978asymptotic}}) to show that there exist a unique
  solution $u_1 \in L^2 [ \Omega, H^1_{\sharp} (Y) / \RR ]$ of the
  cell problem (\ref{eq:cellvect}). Moreover, we observe that (as consequence
  of Lax-Milgram theorem), for every $i \in \NN_N$ and for a.e. $x \in \Omega$
  there exists a unique solution $\chi_i (x, \cdot) \in H^1_{\sharp} (Y) /
  \RR$ of the distributional equation
  \begin{equation}
    - \divy [ A (x, y) \left. \grady \chi_i (x, y) \right) ] \eqs
    \divy A_i (x, y) \label{eq:splittedsoldistrib},
  \end{equation}
  and the stability estimates $\| \chi_i (x, \cdot) \|_{H^1_{\sharp} \left(
  \YY \right)} \leqslant \frac{1}{\alpha} \| A_i \|_{L^{\infty} \left( \Omega
  \times \YY \right)}$ holds a.e. in $\Omega$. Therefore for every $i \in
  \NN_N$ we have $\chi_i \in L^{\infty} [ \Omega, H^1_{\sharp} \left( \YY
  \right) / \RR ]$ so that the unique solution of
  (\ref{eq:splittedsoldistrib}) can be expressed as
  \begin{equation}
    u_1 (x, y) = \gradx u_0 (x) \cdot \tmmathbf{\chi} (x, y)
    \label{eq:explicitchi}
  \end{equation}
  with $\tmmathbf{\chi} (x, y) \assign (\chi_1 (x, y), \chi_2 (x, y), \ldots,
  \chi_N (x, y))$. After that, substituting (\ref{eq:explicitchi}) into
  (\ref{eq:homeq2}) we get the classical homogenized equation:
  \begin{eqnarray}
    f (x) & = & - \divx \left( \left\langle A (x, \cdot) \left( I_N + \grady
    \tmmathbf{\chi} (x, \cdot) \right) \right\rangle_{\YY} \gradx u_0 (x)
    \right) \nonumber\\
    & = & - \divx \left( A_{\hom} (x) \gradx u_0 (x) \right),
    \label{eq:classicalhomequation} 
  \end{eqnarray}
  with
  \begin{equation}
    A_{\hom} (x) \assign \int_{\YY} A (x, y) \left( I_N + \grady
    \tmmathbf{\chi} (x, y) \right) \mathd y. \label{eq:Ahom}
  \end{equation}
  Note that equation (\ref{eq:classicalhomequation}) is well-posed in $H_0^1
  (\Omega)$ since it is easily seen that $A_{\hom}$ is bounded and coercive
  (see {\cite{papanicolau1978asymptotic}}). The proof is complete.
\end{proof}

\section{ Strong Convergence in $H^1 (\Omega)$: A corrector
result\label{sec:correctorH1conv}}
In the classical framework of two-scale convergence, the so-called corrector
results aim to improve the convergence of the solution gradients $\gradx
u_{\varepsilon}$ by adding corrector terms. A typical corrector result has the
effect of transforming a weak convergence result into a strong one
{\cite{Allaire1992,allaire2012shape,papanicolau1978asymptotic}}. In our
context, as we shall see in a moment, the role of the corrector term is
replaced by the average over the unit cell $\YY$ of the family of solutions
$u_{\varepsilon}$ (cfr. Theorem \ref{thm:newhomapproach} for the notations).
We thus get a rigorous justification of the two first term in the asymptotic
expansion (\ref{eq:asymptexpansion}) of the solution $u_{\varepsilon}$ of the
homogenization problem.

\begin{theorem}
  \label{thm:corrH1}For every $\varepsilon \in \RR^+$ let $u_{\varepsilon} \in
  L^2_{\sharp} [ \YY, H^1_0 (\Omega) ]$ be the unique solution of
  the homogenization problem
  (\ref{eq:homogproblem2dot0feq})-(\ref{eq:homogproblem2dot0}), and $(u_0,
  u_1) \in H_0^1 (\Omega) \times L^2 [ \Omega, H^1_{\sharp} (Y) / \RR
  ]$ the unique solution of the homogenized system of equations
  (\ref{eq:homeq1distrib})-(\ref{eq:homeq2}). Then for $\varepsilon
  \rightarrow 0$ we have
  \begin{equation}
    \left\| \langle u_{\varepsilon} \rangle_{\YY} - u_0 \right\|_{H^1
    (\Omega)} \rightarrow 0. \label{eq:strongconvh1}
  \end{equation}
  In particular $\left\langle \gradx u_{\varepsilon} \right\rangle_{\YY} -
  \gradx u_0 \rightarrow 0$ strongly in $L^2 (\Omega)$.
\end{theorem}

\begin{remark}
  Let us recall that in the classical setting and under some more restrictive
  assumptions on the matrix $A$ and on the regularity of the homogenized
  solution $u_0$, it is possible to prove (cfr.
  {\cite{allaire1999boundary,papanicolau1978asymptotic}} that $\|
  u_{\varepsilon} (x) - u (x) - \varepsilon u_1 (x, x / \varepsilon) \|_{H^1
  (\Omega)} \in \mathcal{O} \left( \sqrt{\varepsilon} \right)$. This estimate,
  although generically optimal, is considered to be surprising since one could
  expect to get $\mathcal{O} (\varepsilon)$ if the next order term in the
  ansatz was truly $\varepsilon^2 u_2 (x, x / \varepsilon)$. As is well known,
  this worse-than-expected result is due to the appearance of boundary
  correctors, which must be taken into account to have $\mathcal{O}
  (\varepsilon)$ estimates. On the other hand, in our framework this this
  phenomenon disappears because of $\langle u_1 \rangle_{\YY} = 0$. Indeed, in
  the average, the {\guillemotleft}classical{\guillemotright} first order
  corrector term $u_1$ does not play any role in the asymptotic expansion of
  $u_{\varepsilon}$ given by (\ref{eq:asymptexpansion}), and as we shall see
  in the next section, the first order significant (not null average)
  corrector is the so-called boundary corrector $v_{\varepsilon}$ (cfr.
  {\cite{allaire1999boundary}} and next section), for which we get the more
  natural result $\left\| \left\langle \gradx u_{\varepsilon} - \gradx u_0 -
  \varepsilon \gradx v_{\varepsilon} \right\rangle_{\YY} \right\|_{L^2
  (\Omega)} \in \mathcal{O} (\varepsilon)$.
\end{remark}

\begin{proof}
  Let us observe that using $u_0$ and $u_1$ as test functions in
  (\ref{eq:weakformtemp1}) and (\ref{eq:homeq1}) we get
  \begin{eqnarray}
    \int_{\YY \times \Omega} A (x, y) \left( \gradx u_0 (x) + \grady u_1 (x,
    y) \right) \cdot \grady u_1 (x, y) \mathd x \mathd y & = & 0 
    \label{eq:tempinv1}\\
    \int_{\YY \times \Omega} A (x, y) [ \gradx u_0 (x) + \grady u_1 (x,
    y) ] \cdot \gradx u_0  (x) \mathd x \mathd y & = & \int_{\Omega} f
    (x) u_0 (x) \mathd x.  \label{eq:tempinv2}
  \end{eqnarray}
  We then observe that ($\alpha$ is the ellipticity constant of the matrix
  $A$) for any $u \in L^2_{\sharp} [ \YY, L^2 (\Omega) ]^{}$ one
  has $\left\| \langle u \rangle_{\YY} \right\|_{\Omega} \leqslant \| u
  \|_{\Omega \times \YY}$ and hence, since $\langle u_1 \rangle_{\YY} = 0$ we
  have
  \begin{eqnarray}
    \alpha \left\| \left\langle \gradx u_{\varepsilon} \right\rangle_{\YY} -
    \gradx u_0 \right\|^2_{L^2 (\Omega)} & \eqs & \alpha \left\| \left\langle
    \mathcal{F}_{\varepsilon} \left( \gradx u_{\varepsilon} \right) - \grady
    u_1 - \gradx u_0 \right\rangle_{\YY} \right\|^2_{L^2 (\Omega)} \\
    & \leqslant & \alpha \int_{\Omega \times \YY} \left|
    \mathcal{F}_{\varepsilon} \left( \gradx u_{\varepsilon} \right) - \left(
    \gradx u_0 + \grady u_1 \right) \right|^2 .  \label{eq:corrtemp1}
  \end{eqnarray}
  By the uniformly ellipticity of $A$ and (\ref{eq:corrtemp1}) we continue to
  estimate
  \begin{eqnarray}
    \alpha \left\| \left\langle \gradx u_{\varepsilon} \right\rangle_{\YY} -
    \gradx u_0 \right\|^2_{L^2 (\Omega)} & \leqslant & \int_{\Omega \times
    \YY} A\mathcal{F}_{\varepsilon} \left( \gradx u_{\varepsilon} \right)
    \cdot \mathcal{F}_{\varepsilon} \left( \gradx u_{\varepsilon} \right)
    \nonumber\\
    &  & \qquad + \int_{\Omega \times \YY} A \left( \gradx u_0 + \grady u_1
    \right) \cdot \left( \gradx u_0 + \grady u_1 \right) \nonumber\\
    &  & \hspace{4em} - \int_{\Omega \times \YY} \mathcal{F}_{\varepsilon}
    \left( \gradx u_{\varepsilon} \right) \cdot \left( A + A^{\T} \right)
    \left( \grady u_1 + \gradx u_0 \right) \\
    & = & \int_{\Omega \times \YY} f (x) u_{\varepsilon} (x, y) \mathd x
    \mathd y + \int_{\Omega \times \YY} A \left( \gradx u_0 + \grady u_1
    \right) \cdot \left( \gradx u_0 + \grady u_1 \right) \nonumber\\
    &  & \hspace{4em} - \int_{\Omega \times \YY} \mathcal{F}_{\varepsilon}
    \left( \gradx u_{\varepsilon} \right) \cdot \left( A + A^{\T} \right)
    \left( \grady u_1 + \gradx u_0 \right), 
  \end{eqnarray}
  the second equality being a consequence of the fact that $u_{\varepsilon}$
  is the solution of the problem
  (\ref{eq:homogproblem2dot0feq})-(\ref{eq:homogproblem2dot0}). Taking into
  account (\ref{eq:tempinv1}) and (\ref{eq:tempinv2}) we then get
  \begin{eqnarray}
    \alpha \left\| \left\langle \gradx u_{\varepsilon} \right\rangle_{\YY} -
    \gradx u_0 \right\|^2_{L^2 (\Omega)} & \leqslant & \int_{\Omega \times
    \YY} f (x) u_{\varepsilon} (x, y) \mathd x \mathd y + \int_{\Omega} f (x)
    u_0 (x) \mathd x \nonumber\\
    &  & \quad - \int_{\Omega \times \YY} \mathcal{F}_{\varepsilon} \left(
    \gradx u_{\varepsilon} \right) \cdot \left( A + A^{\T} \right) \left(
    \grady u_1 + \gradx u_0 \right) . 
  \end{eqnarray}
  Since $(A + A^T) \left( \grady u_1 + \gradx u_0 \right) \in L^2_{\sharp}
  [ \YY, L^2 (\Omega) ]$, it is a test function for the two-scale
  convergence, so that (again from (\ref{eq:tempinv1}) and
  (\ref{eq:tempinv2}))
  \begin{eqnarray}
    - \lim_{\varepsilon \rightarrow 0} \int_{\Omega \times \YY}
    \mathcal{F}_{\varepsilon} \left( \gradx u_{\varepsilon} \right) \cdot
    \left( A + A^{\T} \right) \left( \grady u_1 + \gradx u_0 \right)  & = & -
    2 \int_{\Omega \times \YY} A \left( \gradx u_0 + \grady u_1 \right) \cdot
    \left( \gradx u_0 + \grady u_1 \right)  \nonumber\\
    & = & - 2 \int_{\Omega} f (x) u_0 (x) \mathd x. 
  \end{eqnarray}
  Finally, to infer (\ref{eq:strongconvh1}), we simply observe that due to the
  $\YY$ periodicity of $u_{\varepsilon}$ one has
  \begin{equation}
    \int_{\Omega \times \YY} f (x) u_{\varepsilon} (x, y) \mathd x \mathd y
    \eqs \int_{\Omega \times \YY} f (x) \mathcal{F} (u_{\varepsilon}) (x, y)
    \mathd x \mathd y
  \end{equation}
  with $u_{\varepsilon} \twoheadrightarrow u_0$. The proof is completed.
\end{proof}

\section{Higher Order Correctors: Boundary Layers}\label{sec:6bl}

In what follows assume that the matrix of diffusion coefficients $A$ is
symmetric and depends on the {\guillemotleft}periodic
variable{\guillemotright} only, i.e. $A \in L^{\infty}_{\sharp} \left( \YY
\right)$, $A = A^{\T}$ and of course $A$ uniformly elliptic with $\alpha > 0$
as constant of ellipticity. By the uniqueness of the solution of the cell
problem (\ref{eq:splittedsoldistribcell}) it is easily seen that in these
hypotheses also the vector of correctors (see Definition
\ref{def:classhomandcorr}) depends on the {\guillemotleft}periodic
variable{\guillemotright} only, i.e. $\tmmathbf{\chi} \in [ H^1_{\sharp}
\left( \YY \right) ]^N$.

In the previous section (see Theorem \ref{thm:corrH1}) we have seen that the
sequence of the averaged solutions $\langle u_{\varepsilon} \rangle_{\YY}$
strongly converge to $u_0$ in $H^1 (\Omega)$, i.e. that $\left\| \left\langle
\gradx u_{\varepsilon} - \gradx u_0 \right\rangle \right\|_{H^1 (\Omega)} \in
\mathcal{O} (1)$. To have higher order estimates, especially near the boundary
of $\Omega$, one has to introduce supplementary terms, called {\tmem{boundary
layers}} {\cite{lions1981some}}, which roughly speaking aim to compensate the
fast oscillation of the family of solutions $u_{\varepsilon}$ near the
boundary $\partial \Omega$. More precisely, in this section we show that under
suitable hypotheses one has
\begin{equation}
  \left\| \left\langle \gradx u_{\varepsilon} - \gradx u_0 - \varepsilon
  \gradx v_{\varepsilon} \right\rangle_{\YY} \right\|_{L^2 (\Omega)} \in
  \mathcal{O} (\varepsilon), \label{eq:higherordercorr1}
\end{equation}
where $v_{\varepsilon}$ is the solution of the boundary layer problem:
\begin{eqnarray}
  \divx \left( A \left( y + \frac{x}{\varepsilon} \right) \gradx
  v_{\varepsilon} (x, y) \right) & = \hspace{1.2em} \; 0 & \text{ in } \Omega
  \times \YY  \label{eq:boundarylayer1D}\\
  v_{\varepsilon} (x, y) & = \hspace{1.2em} \; u_1 (x, y + x / \varepsilon) &
  \text{ on } \partial \Omega \times \YY .  \label{eq:boundarylayer1Dbc}
\end{eqnarray}
We also investigate the validity of the following stronger estimate
\begin{equation}
  \left\| \left\langle \gradx u_{\varepsilon} - \gradx u_0 \right\rangle_{\YY}
  \right\|_{L^2 (\Omega)} \in \mathcal{O} (\varepsilon) .
  \label{eq:higherordercorr2}
\end{equation}
Quite remarkably, as we are going to show in the next subsection, in the
one-dimensional case the stronger estimate (\ref{eq:higherordercorr2}) holds
under the same hypotheses of the weaker estimate (\ref{eq:higherordercorr1}).

\subsection{Higher Order Correctors in dimension one}

In the one-dimensional setting $\YY = [0, 1]$ and $\Omega \subseteq \RR$ is an
open interval: $\Omega \assign (0, \omega)$ with $\omega > 0$. We then denote
by $a \in L^{\infty}_{\sharp} \left( \YY \right)$ the unique coefficient of
the matrix valued function $A$. Finally for the generic {\guillemotleft}1D
function{\guillemotright} function $u \in L^2_{\sharp} [ \YY, H^1_0
(\Omega) ]$ we shall denote by $u' \in L^2_{\sharp} [ \YY, L^2
(\Omega) ]$ the weak derivative with respect to the $x$ variable.

\begin{theorem}
  \label{thm:BoundaryLayers1D}Let $(u_0, u_1) \in H_0^1 (\Omega) \times L^2
  [ \Omega, H^1_{\sharp} (Y) / \RR ]$ be the unique solution of the
  homogenized system of equations (\ref{eq:homeq1distrib})-(\ref{eq:homeq2}).
  The following estimate holds
  \begin{equation}
    \left\| \langle u'_{\varepsilon} \rangle_{\YY} - u_0' \right\|_{L^2
    (\Omega)} \leqslant 2 \varepsilon \cdot \frac{| a |_{\infty} |
    \chi_{\infty} |}{\alpha} (| u_0' |_{\infty} + \| u_0'' \|_{L^2 (\Omega)})
    . \label{eq:1Dstrongestimate}
  \end{equation}
\end{theorem}

We will need the following two lemmas

\begin{lemma}
  For any $\varepsilon > 0$ let $u_{\varepsilon} \in L^2_{\sharp} [ \YY,
  H^1_0 (\Omega) ]$ be the unique solution of the problem
  (\ref{eq:homogproblem2dot0feq})-(\ref{eq:homogproblem2dot0}). Define the
  {\mybold{error function}}
  \begin{equation}
    e_{\varepsilon} (x, y) \assign u_{\varepsilon} (x, y) - u_0 (x) -
    \varepsilon \left. [ u_1 \left( x, y + \frac{x}{\varepsilon} \right)
    - v_{\varepsilon} (x, y) \right. ],
  \end{equation}
  where $v_{\varepsilon} \in L^2_{\sharp} [ \YY, H^1 (\Omega) ]$ is
  the unique solution of the boundary layer problem
  (\ref{eq:boundarylayer1D})-(\ref{eq:boundarylayer1Dbc}). The following
  estimate holds:
  \begin{equation}
    \| e_{\varepsilon}' \|_{L^2 \left( \Omega \times \YY \right)} \leqslant
    \varepsilon \cdot \frac{| a |_{\infty} | \chi_{\infty} |}{\alpha} \| u_0''
    \|_{L^2 (\Omega)} \in \mathcal{O} (\varepsilon) .
    \label{eq:fundest1derrorfunction1}
  \end{equation}
\end{lemma}

\begin{proof}
  In the 1D setting, the homogenized equation (\ref{eq:homogenizedpdej1}) read
  as $a_{\hom} u''_0 (x) = - f (x)$ with $a_{\hom} \assign \langle a^{- 1}
  (\cdot) \rangle^{- 1}_{\YY} > 0$ and therefore $u_0 \in H_0^2 (\Omega)$.
  Indeed, as a consequence of Theorem \ref{thm:newhomapproach} (see eq.
  (\ref{eq:linkclassicaland2scale})), the unique solution $u_1 \in L^2 [
  \Omega, H^1_{\sharp} (Y) / \RR ]$ of (\ref{eq:homeq1distrib}) can be
  expressed in the tensor product form $u_1 (x, y) = \chi (y) u_0' (x)$, where
  $\chi$ is the unique (null average) solution in $H^1_{\sharp} \left( \YY
  \right) / \RR$ of (\ref{eq:splittedsoldistribcell}). A direct integration of
  the cell equation (\ref{eq:splittedsoldistribcell}) leads to (taking into
  account the periodicity of $u_1$ and averaging over {\YY}) $a (y) (1 +
  \partial_y \chi (y)) \eqs a_{\hom}$ with $a_{\hom} \assign \langle a (y) (1
  + \partial_y \chi (y)) \rangle_{\YY} = \langle a^{- 1} (\cdot) \rangle^{-
  1}_{\YY}$.
  
  Since $a_{\hom} u''_0 (x) = - f (x)$ from (\ref{eq:homogproblem2dot0feq}) we
  get $[a_{\hom} u_0' (x)]' = [ a \left( y + \frac{x}{\varepsilon}
  \right) u_{\varepsilon}' (x, y) ]'$. Hence, taking into account the
  equation satisfied by $v_{\varepsilon}$, a direct computation shows that for
  a.e. $y \in \YY$ the function $e_{\varepsilon} (\cdot, y)$ satisfies the
  distributional equation
  \begin{equation}
    - \left( a \left( y + \frac{x}{\varepsilon} \right) e_{\varepsilon}' (x,
    y) \right)' \eqs \varepsilon F_{\varepsilon}' (x, y) \quad \lo{\tmop{in}}
    \mathcal{D}' (\Omega), \label{eq:fundequality1Deps}
  \end{equation}
  with $F_{\varepsilon} (x, y) \assign a \left( y + \frac{x}{\varepsilon}
  \right)_{_{_{_{_{}}}}} \chi \left( y + \frac{x}{\varepsilon} \right) u_0''
  (x)$. For every $\varphi_{\varepsilon} \in L^2_{\sharp} [ \YY, H^1_0
  (\Omega) ]$, the variational form in $L^2_{\sharp} [ \YY, H^1_0
  (\Omega) ]$ of (\ref{eq:fundequality1Deps}) reads as
  \begin{equation}
    \int_{\Omega \times \YY} a \left( y + \frac{x}{\varepsilon} \right)
    e_{\varepsilon}' (x, y) \cdot \varphi_{\varepsilon}' (x, y) \mathd x
    \mathd y \eqs - \varepsilon \int_{\Omega \times \YY} F_{\varepsilon} (x,
    y) \cdot \varphi_{\varepsilon}' (x, y) \mathd x \mathd y.
    \label{eq:var1Derrorfunction}
  \end{equation}
  Since $e_{\varepsilon} \in L^2_{\sharp} [ \YY, H^1_0 (\Omega) ]$
  evaluating the variational equation (\ref{eq:var1Derrorfunction}) on the
  test function $\varphi_{\varepsilon} (x, y) \assign e_{\varepsilon} (x, y)$
  and recalling that $a \geqslant \alpha$ we finish with
  (\ref{eq:fundest1derrorfunction1}).
\end{proof}

\begin{lemma}
  \label{lemma:uniformboundveps}Let $v_{\varepsilon} \in L^2_{\sharp} [
  \YY, H^1 (\Omega) ]$ solve the boundary value problem
  (\ref{eq:boundarylayer1D})-(\ref{eq:boundarylayer1Dbc}). Then the following
  uniform estimate (with respect to $\varepsilon$) holds:
  \begin{equation}
    \| v_{\varepsilon}' (x, y) \|_{L^2 \left( \Omega \times \YY \right)}
    \leqslant \frac{2}{\alpha} | a |_{\infty} | \chi |_{\infty} | u_0'
    |_{\infty}  \label{eq:fundest1derrorfunction2} .
  \end{equation}
\end{lemma}

\begin{proof}
  Let us integrate (\ref{eq:boundarylayer1D}). We get $v_{\varepsilon}' (x,
  y) \eqs c_{\varepsilon} (y) a^{- 1} (y + x / \varepsilon)$ for some
  measurable real function $c_{\varepsilon}$. Taking into account boundary
  conditions (\ref{eq:boundarylayer1Dbc}), we compute
  \begin{equation}
    c_{\varepsilon} (y) = \frac{\chi (y + \omega / \varepsilon) u_0' (\omega)
    - \chi (y) u_0' (0)}{| \Omega | \langle a^{- 1} (y + \cdot / \varepsilon)
    \rangle_{\Omega}} .
  \end{equation}
  Next we note that $\left\langle a^{- 1} \left( y + \frac{\cdot}{\varepsilon}
  \right) \right\rangle_{\Omega}^{- 1} \leqslant | a |_{\infty}$ for a.e. $y
  \in \RR$. Hence, observing that since $f \in L^2 (\Omega)$ one has $u_0 \in
  W^{1, \infty} (\Omega)$, we finish with the estimate $\alpha | \Omega | |
  v_{\varepsilon}' (x, y) | \leqslant 2 | a |_{\infty} | \chi |_{\infty} |
  u_0' |_{\infty}$ from which (\ref{eq:fundest1derrorfunction2}) immediately
  follows.
\end{proof}

We can now prove Theorem \ref{thm:BoundaryLayers1D}.

\begin{proof}[of Theorem \ref{thm:BoundaryLayers1D}]
  Observing that $\left\langle [ u_1 \left( x, y + \frac{x}{\varepsilon}
  \right) ]' \right\rangle_{\YY} \eqs 0$ we compute
  \begin{eqnarray}
    \left\| \langle u'_{\varepsilon} \rangle_{\YY} - u_0' \right\|_{L^2
    (\Omega)} & \eqs & \left\| \left\langle u'_{\varepsilon} (x, y) - u_0 (x)
    - \varepsilon [ u_1 \left( x, y + \frac{x}{\varepsilon} \right)
    ]' \right\rangle_{\YY}' \right\|_{L^2 (\Omega)} \\
    & \leqslant & \left\| \langle e'_{\varepsilon} (x, y) \rangle_{\YY}
    \right\|_{\Omega} + \varepsilon \left\| \langle v_{\varepsilon}' (x, y)
    \rangle_{\YY} \right\|_{L^2 (\Omega)} \\
    & \leqslant & \| e'_{\varepsilon} \|_{L^2 \left( \Omega \times \YY
    \right)} + \varepsilon \| v_{\varepsilon}' \|_{L^2 \left( \Omega \times
    \YY \right)} . 
  \end{eqnarray}
  Hence taking into account estimates (\ref{eq:fundest1derrorfunction1}) and
  (\ref{eq:fundest1derrorfunction2}) we get the result.
\end{proof}

\subsection{Higher Order Correctors in $N$ dimensions}

This section is devoted to the proof of estimate (\ref{eq:higherordercorr1}).

\begin{theorem}
  \label{thm:BoundaryLayersND}Let $(u_0, u_1) \in H_0^1 (\Omega) \times L^2
  [ \Omega, H^1_{\sharp} (Y) / \RR ]$ be the unique solution of the
  homogenized system of equations (\ref{eq:homeq1distrib})-(\ref{eq:homeq2}).
  Define the error function by the position
  \begin{equation}
    e_{\varepsilon} (x, y) \assign u_{\varepsilon} (x, y) - u_0 (x) -
    \varepsilon \left. [ u_1 \left( x, y + \frac{x}{\varepsilon} \right)
    - v_{\varepsilon} (x, y) \right. ],
  \end{equation}
  $v_{\varepsilon} \in L^2_{\sharp} [ \YY, H^1 (\Omega) ]$ being
  the unique solution of the boundary layer problem
  (\ref{eq:boundarylayer1D})-(\ref{eq:boundarylayer1Dbc}). If $u_0 \in H^2
  (\Omega)$ then $\left\| \left\langle \gradx e_{\varepsilon}
  \right\rangle_{\YY} \right\|_{\Omega} \in \mathcal{O} (\varepsilon)$. More
  precisely, the following estimate holds
  \begin{equation}
    \left\| \left\langle \gradx u_{\varepsilon} \right\rangle_{\YY} - \gradx
    u_0 + \varepsilon \left\langle \gradx v_{\varepsilon} \right\rangle_{\YY}
    \right\|_{L^2 (\Omega)} \leqslant \varepsilon c_A  \| u_0 \|_{H^2
    (\Omega)}, \label{eq:NDstrongestimate}
  \end{equation}
  for a suitable constant $c_{\alpha} > 0$ depending on the matrix $A$ only.
\end{theorem}

\begin{proof}
  Let us set $u_1^{\varepsilon} (x, y) \assign u_0 (x) + \varepsilon u_1 (x,
  y)$, where $u_1 (x, y) = \gradx u_0 (x) \cdot \tmmathbf{\chi} (y)$ as shown
  in Theorem \ref{thm:newhomapproach}. We have (let us denote by
  $\mathcal{H}_x \assign \gradx \gradx$ the partial hessian operator)
  \begin{equation}
    \gradx [ u_1^{\varepsilon} \left( x, y + \frac{x}{\varepsilon}
    \right) ] \eqs [ I + \grady \tmmathbf{\chi} \left( y +
    \frac{x}{\varepsilon} \right) ] \gradx u_0 + \varepsilon
    \mathcal{H}_x [u_0] (x) \tmmathbf{\chi} \left( y + \frac{x}{\varepsilon}
    \right) .
  \end{equation}
  Hence
  \begin{equation}
    A_{\hom} \gradx u_0 (x) - A \left( y + \frac{x}{\varepsilon} \right)
    \gradx [ u_1^{\varepsilon} \left( x, y + \frac{x}{\varepsilon}
    \right) ] \eqs \mathcal{A}_0 \left( y + \frac{x}{\varepsilon}
    \right) \gradx u_0 - \varepsilon \tmmathbf{h} \left( x, y +
    \frac{x}{\varepsilon} \right), \label{eq:jjtemp1}
  \end{equation}
  where, for notational convenience, we have introduce the functions
  \begin{equation}
    \tmmathbf{h} (x, y) \assign A (y) \mathcal{H}_x [u_0] (x)
    \tmmathbf{\chi} (y) \quad, \quad \mathcal{A}_0 (y) \assign A_{\hom} -
    a_{\hom} (y)
  \end{equation}
  with $a_{\hom} (y) \assign A (y) [ I + \grady \tmmathbf{\chi} (y)
  ]$. Let us note that $\langle \mathcal{A}_0 \rangle_{\YY} = 0$,
  because $A_{\hom} = \langle a_{\hom} (\cdot) \rangle_{\YY}$. By taking the
  distributional divergence of both members of the previous equation
  (\ref{eq:jjtemp1}), recalling that $\divx \left( A \left( y +
  \frac{x}{\varepsilon} \right) \gradx v_{\varepsilon} (x, y) \right) = 0$,
  that due to (\ref{eq:homogproblem2dot0feq}) and (\ref{eq:homogenizedpdej1})
  one has
  \begin{equation}
    \divx [ A \left( y + \frac{x}{\varepsilon} \right) \gradx
    u_{\varepsilon} (x, y) ] \eqs \divx \left( A_{\hom} (x) \gradx u_0
    (x) \right),
  \end{equation}
  and that $v_{\varepsilon} \in L^2_{\sharp} [ \YY, H^1 (\Omega) ]$
  is the solution of the boundary layer problem
  (\ref{eq:boundarylayer1D})-(\ref{eq:boundarylayer1Dbc}), we get
  \begin{equation}
    \divx \left( A \left( y + \frac{x}{\varepsilon} \right) \gradx
    [e_{\varepsilon} (x, y)]  \right) \eqs \divx \left( \mathcal{A}_0 \left( y
    + \frac{x}{\varepsilon} \right) \gradx u_0 \right) - \varepsilon \divx
    [ \tmmathbf{h} \left( x, y + \frac{x}{\varepsilon} \right) ] .
    \label{eq:fundequality1DepsN}
  \end{equation}
  Next, let us recall that in the space $L^2_{\tmop{sol}} \left( \YY \right)$
  of solenoidal and periodic vector fields, defined by the position
  $L^2_{\tmop{sol}} \left( \YY \right) \assign \left\{ \tmmathbf{p} \in L^2
  \left( \YY \right) \; : \; \divv \tmmathbf{p} (y) = 0 \right\}$ the
  following Helmholtz-Hodge decomposition holds (cfr. {\cite{Jikov1994}}): if
  $\tmmathbf{p} \in L^2_{\tmop{sol}} \left( \YY \right)$ there exists a
  skew-symmetric matrix $\tmmathbf{\omega} \assign \left( \tmmathbf{\omega}^1
  \; | \; \tmmathbf{\omega}^2 \; | \; \cdots \; | \; \tmmathbf{\omega}^N
  \right) \in [ H^1_{\sharp} \left( \YY \right) ]^{N \times N}$
  such that
  \begin{equation}
    \langle \tmmathbf{\omega} \rangle_{\YY} =\tmmathbf{0} \quad, \quad
    \tmmathbf{p} \eqs \langle \tmmathbf{p} \rangle_{\YY} + \sum_{j = 1}^N
    \partial_j \tmmathbf{\omega}^j = \langle \tmmathbf{p} \rangle_{\YY} +
    \curl \tmmathbf{\omega},
  \end{equation}
  with $\curl : \tmmathbf{\omega} \mapsto \curl \tmmathbf{\omega} \assign
  \partial_1 \tmmathbf{\omega}^1 + \cdots + \partial_N \tmmathbf{\omega}^N$.
  Note that $\divy \mathcal{A}_0 (y) =\tmmathbf{0}$ because $\mathcal{A}_0$
  solves the cell equation (\ref{eq:splittedsoldistribcell}). On the other
  hand, $\langle \mathcal{A}_0 \rangle_{\YY} =\tmmathbf{0}$ and therefore due
  to the Helmholtz-Hodge decomposition there exist skew-symmetric matrices
  $(\tmmathbf{\omega}_i)_{i \in \NN_N} \in [ H^1_{\sharp} \left( \YY
  \right) ]^{N \times N}$ such that $\mathcal{A}_0 (y) \mathbf{e}_i =
  \curl \tmmathbf{\omega}_i (y)$ for every $i \in \NN_N$. From the scaling
  relation
  \begin{equation}
    \varepsilon \cdot \curlx [\tmmathbf{\omega}_i (y + x / \varepsilon)] =
    \curl \tmmathbf{\omega}_i (y + x / \varepsilon),
  \end{equation}
  recalling that for any $g \in H^1_{\sharp} \left( \YY \right),
  \tmmathbf{\omega} \in [ H^1_{\sharp} \left( \YY \right) ]^{N
  \times N}$ one has $\curl (g\tmmathbf{\omega}) =\tmmathbf{\omega} \nabla g +
  g \curl \tmmathbf{\omega}$ in $\mathcal{D}'$, we have
  \begin{eqnarray}
    \mathcal{A}_0 \left( y + \frac{x}{\varepsilon} \right) \gradx u_0 (x) &
    \eqs & \varepsilon \sum_{i \in \NN_N} \partial_i u_0 (x) \curlx
    [\tmmathbf{\omega}_i (y + x / \varepsilon)] \nonumber\\
    & \eqs & \varepsilon \sum_{i \in \NN_N} \curlx [\partial_i u_0 (x)
    \tmmathbf{\omega}_i (y + x / \varepsilon)] - \varepsilon \sum_{i \in
    \NN_N} \tmmathbf{\omega}_i (y + x / \varepsilon) \partial_i \gradx u_0 (x)
    \nonumber\\
    & \eqs & \varepsilon \sum_{i \in \NN_N} \curlx [\partial_i u_0 (x)
    \tmmathbf{\omega}_i (y + x / \varepsilon)] - \varepsilon \tmmathbf{\eta}
    (x, y + x / \varepsilon), 
  \end{eqnarray}
  with $\tmmathbf{\eta} (x, y) \assign \sum_{i \in \NN_N} \tmmathbf{\omega}_i
  (y) \partial_i \gradx u_0 (x) \in L^2_{\sharp} [ \YY, L^2 (\Omega)
  ]$ and $\langle \tmmathbf{\eta} \rangle_{\YY} =\tmmathbf{0}$. Passing
  to the divergence in the previous relations, we get $\divx \left(
  \mathcal{A}_0 \left( y + \frac{x}{\varepsilon} \right) \gradx u_0 (x)
  \right) = \varepsilon \divx (\tmmathbf{\eta} (x, y + x / \varepsilon))$.
  Hence, equation (\ref{eq:fundequality1DepsN}) simplifies to
  \begin{equation}
    \divx \left( A \left( y + \frac{x}{\varepsilon} \right) \gradx
    [e_{\varepsilon} (x, y)]  \right) \eqs - \varepsilon \divx
    \tmmathbf{F}_{\varepsilon} (x, y) \quad \text{in } \mathcal{D}' \left(
    \Omega \times \YY \right), \label{eq:fundequality1Depscurl}
  \end{equation}
  with $\tmmathbf{F}_{\varepsilon} (x, y) \assign \tmmathbf{\eta} (x, y + x /
  \varepsilon) +\tmmathbf{h} \left( x, y + \frac{x}{\varepsilon} \right) \in
  L^2_{\sharp} [ \YY, L^2 (\Omega) ]^N$ and $\{
  \tmmathbf{F}_{\varepsilon} \}_{\varepsilon \in \RR^+}$ a bounded subset of
  $L^2_{\sharp} [ \YY, L^2 (\Omega) ]^N$. The previous equation
  (\ref{eq:fundequality1Depscurl}) reads in variational form as
  \begin{equation}
    \int_{\Omega \times \YY} A \left( y + \frac{x}{\varepsilon} \right) \gradx
    e_{\varepsilon} (x, y) \cdot \gradx \varphi_{\varepsilon} (x, y) \mathd x
    \mathd y \eqs - \varepsilon \int_{\Omega \times \YY}
    \tmmathbf{F}_{\varepsilon} (x, y) \cdot \gradx \varphi_{\varepsilon} (x,
    y) \mathd x \mathd y, \label{eq:fundepsestimatenD}
  \end{equation}
  for any $\varphi_{\varepsilon} \in L^2_{\sharp} [ \YY, H_0^1 (\Omega)
  ]$. Since $v_{\varepsilon}$ solves the boundary layer problem
  (\ref{eq:boundarylayer1D})-(\ref{eq:boundarylayer1Dbc}), we have
  $e_{\varepsilon} \in L^2_{\sharp} [ \YY, H_0 (\Omega) ]$ and
  therefore, testing (\ref{eq:fundepsestimatenD}) on $e_{\varepsilon}$ we
  finish, for some suitable constant $c_A > 0$ depending on $A$ only, with
  (\ref{eq:NDstrongestimate}).
\end{proof}

\section{Weak two-scale compactness for $\SStwo$-valuedHarmonic
maps}\label{sec:homharmonicmaps}

The aim of this section is to prove a weak two-scale compactness result for
$\SStwo$-valued harmonic maps, and make some remarks which point out possible
weaknesses of this alternative notion of two-scale convergence.

In what follows \ $\Omega$ is a bounded and Lipschitz domain of $\RR^3$ and we
shall make use of the following notations: $W (\Omega) \assign L^{\infty}
(\Omega) \cap H^1 (\Omega)$ and $W_0 (\Omega) \assign L^{\infty} (\Omega) \cap
H^1_0 (\Omega)$.

\subsection{Harmonic maps equation}

We want to focus on the homogenization of the family of harmonic map equations
arising as the Euler-Lagrange equations associated to the family of Dirichlet
energy functionals
\begin{equation}
  \mathcal{E}_{\varepsilon} (\tmmathbf{u}_{\varepsilon}) \assign \int_{\Omega
  \times \YY} a_{\varepsilon} (x, y) \gradx \tmmathbf{u}_{\varepsilon} (x, y)
  \cdot \gradx \tmmathbf{u}_{\varepsilon} (x, y) \mathd x \mathd y \quad,
  \quad a_{\varepsilon} (x, y) \assign a \left( x, y + \frac{x}{\varepsilon}
  \right), \label{eq:dirichletenfunc}
\end{equation}
all defined in $L^{\infty}_{\sharp} [ \YY, W \left( \Omega, \SStwo
\right) ]^3$. Here, as usual, the coefficient $a \in L^{\infty}_{\sharp}
\left( \YY, L^{\infty} (\Omega) \right)$ is a positive function bounded from
below by some positive constant. The stationary condition on
$\mathcal{E}_{\varepsilon}$ with respect to tangential variations in
$L^{\infty}_{\sharp} [ \YY, W_0 (\Omega) ]^3$ conducts to the
equation of harmonic maps
\begin{equation}
  \int_{\Omega \times \YY} a_{\varepsilon} (x, y) \gradx
  \tmmathbf{u}_{\varepsilon} (x, y)  \gradx \tmmathbf{\eta}_{\varepsilon} (x,
  y) \mathd x \mathd y = 0 \label{eq:statharmmap}
\end{equation}
which must be satisfied for every $\tmmathbf{\eta}_{\varepsilon} \in
L^{\infty}_{\sharp} [ \YY, W_0 (\Omega) ]^3$ such that
$\tmmathbf{\eta}_{\varepsilon} (x, y) \in T_{\tmmathbf{u}_{\varepsilon} (x,
y)} \SStwo$ a.e. in $\Omega \times \YY$.

\begin{theorem}
  \label{thm:homogenizationhmaps}For every $\varepsilon \in \RR^+$ let
  $\tmmathbf{u}_{\varepsilon} \in L^{\infty}_{\sharp} [ \YY, W \left(
  \Omega, \SStwo \right) ]^3$ be a solution of the harmonic map equation
  (\ref{eq:statharmmap}). If $(\tmmathbf{u}_{\varepsilon})_{\varepsilon \in
  \RR^+} \twoheadrightarrow \tmmathbf{u}_0$ weakly in $\mathrm{L^2_{\sharp}
  [ \YY, H^1 (\Omega) ]}^3$ and $\tmmathbf{u}_0$ takes values on
  $\SStwo$, then $\tmmathbf{u}_0$ is still an harmonic map. More precisely,
  $\tmmathbf{u}_0 \in W \left( \Omega, \SStwo \right)^3$ satisfies the
  following homogenized harmonic map equation
  \begin{equation}
    \int_{\Omega} A_{\hom} (x)_{_{_{_{}}}} \nabla \tmmathbf{u}_0 (x) \nabla
    \tmmathbf{\varphi} (x) \mathd x \eqs 0 \quad \forall \tmmathbf{\varphi}
    \in T_{\tmmathbf{u}_0} \SStwo  \label{eq:homogenizedharmonicmapeq}
  \end{equation}
  in which
  \begin{equation}
    A_{\hom} (x) \assign \int_{\YY} a (x, y) \left( I + \grady \tmmathbf{\chi}
    (x, y) \right) \mathd y,
  \end{equation}
  and $\tmmathbf{\chi} \assign (\chi_1, \chi_2, \chi_3) \in L^2 [ \Omega,
  H^1_{\sharp} \left( \YY \right) ]^3$ is the unique null average
  solution of the cell problems ($i \in \NN_3$)
  \begin{equation}
    \divy \left( a (x, y) \left( \grady \chi_{i_{}} (x, y) + e_i \right)
    \right) \eqs 0. \label{eq:cellharmonic1}
  \end{equation}
\end{theorem}

\begin{remark}
  \label{rm:db1}In stating Theorem \ref{thm:homogenizationhmaps} we have
  assumed that the weak limit $\tmmathbf{u}_0$ still takes values on the unit
  sphere of $\RR^3$. Indeed, and this is a drawback of the alternative
  two-scale notion, although the introduction of the $y$ variable in
  (\ref{eq:statharmmap}) overcomes the problem of the admissibility of the
  coefficient $a_{\varepsilon}$, it introduces a loss of compactness into the
  family of energy functionals $\mathcal{E}_{\varepsilon}$ defined in
  (\ref{eq:dirichletenfunc}). Indeed, in the space $\mathrm{L^2_{\sharp}
  [ \YY, H^1 \left( \Omega, \SStwo \right) ]}^3$,
  Rellich--Kondrachov theorem does not apply, and therefore any uniform bound
  on the family $\mathcal{E}_{\varepsilon}$ does not assure compactness of
  minimizing sequences.
\end{remark}

\begin{remark}
  The same result still holds, with minor modifications, if we replace
  $\SStwo$ with $\SSn$. Moreover an analogue result holds if one replace the
  energy density $a_{\varepsilon} \left| \gradx \tmmathbf{u}_{\varepsilon}
  \right|^2$ with the energy density $\sum_{i \in \NN_3} A_{i, \varepsilon}
  \gradx u_{i, \varepsilon} \cdot \gradx u_{i, \varepsilon}$ in which every
  $A_{\varepsilon, i}$ is a definite positive symmetric matrix. On the other
  hand, the proof does not work anymore when the image manifold is arbitrary.
  Indeed, for $\SSn$ valued maps, we can exploit a result of {\tmname{Chen}}
  {\cite{chen1989weak}} which permits to equivalently write the Euler-Lagrange
  equation (\ref{eq:statharmmap}) as an equation in divergence form.
  Unfortunately, this conservation law heavily relies on the invariance under
  rotations of Dirichlet energy for maps into $\SSn$. As a matter of fact,
  when the target manifold is arbitrary, even the less general problem
  concerning weak compactness for weakly harmonic maps remains open
  {\cite{Helein2002harmonic}}.
\end{remark}

We shall make use of the following Lemma which, although more than sufficient
for addressing our problem, can still be rephrased to cover more general
situations. Note that an equivalent result, in the context of classical
two-scale convergence, has already been proved in
{\cite{alouges2015homogenization}}.

\begin{lemma}
  \label{Lemma:2scalemanifold}Let $\mathcal{M} \subset \RR^N$ be a regular
  closed orientable hypersurface, and let
  $(\tmmathbf{u}_{\varepsilon})_{\varepsilon \in \RR^+}$ be a family of
  $L^2_{\sharp} [ \YY, H^1 (\Omega) ]^N$ vector fields such that
  $\tmmathbf{u}_{\varepsilon} (x, y) \in \mathcal{M}$ a.e. in $\Omega \times
  \YY$. If for some $\tmmathbf{u}_0 \in L^2_{\sharp} [ \YY, L^2 (\Omega)
  ]^N$, $\tmmathbf{\xi} \in L^2_{\sharp} [ \YY, L^2 (\Omega)
  ]^{N \times N}$ one has
  \begin{equation}
    \tmmathbf{u}_{\varepsilon} \twoheadrightarrow \tmmathbf{u}_0 \quad
    \text{strongly in } L^2_{\sharp} [ \YY, L^2 (\Omega) ]^N \quad,
    \quad \gradx \tmmathbf{u}_{\varepsilon} \twoheadrightarrow \tmmathbf{\xi}
    \quad \text{in } L^2_{\sharp} [ \YY, L^2 (\Omega) ]^{N \times
    N},
  \end{equation}
  then $\tmmathbf{u}_0 (x, y) = \langle \tmmathbf{u}_0 (x, \cdot)
  \rangle_{\YY}$, i.e. the two-scale limit $\tmmathbf{u}_0$ does not depends
  on the $y$ variable. Moreover there exists an element $\tmmathbf{u}_1 \in
  L^2 [ \Omega, H^1_{\sharp} \left( \YY \right) ]^N$ such that
  \begin{equation}
    \gradx \tmmathbf{u}_{\varepsilon} \twoheadrightarrow \left( \gradx
    \tmmathbf{u}_0 + \grady \tmmathbf{u}_1 \right) \text{ weakly in
    $L^2_{\sharp} [ \YY, L^2 (\Omega) ]^N$}
    \label{eq:statementalreadyknown}
  \end{equation}
  with $\tmmathbf{u}_0 (x) \in \mathcal{M}$ and $\tmmathbf{u}_1 (x, y) \in
  T_{\tmmathbf{u}_0 (x)} \mathcal{M}$ for a.e. $(x, y) \in \Omega \times \YY$.
\end{lemma}

\begin{remark}
  Here, as already observed in Remark \ref{rm:db1}, we have to assume strong
  two-scale convergence since the boundedness of the family
  $(\tmmathbf{u}_{\varepsilon})_{\varepsilon \in \RR^+}$ in $L^2_{\sharp}
  [ \YY, H^1 (\Omega) ]^N$ does not imply strong convergence in
  $L^2_{\sharp} [ \YY, L^2 (\Omega) ]^N$ of a suitable subsequence,
  which is an essential requirement in order to prove that the limit function
  $\tmmathbf{u}_0$ takes values on $\mathcal{M}$.
\end{remark}

\begin{proof}
  Since $\tmmathbf{u}_{\varepsilon} \twoheadrightarrow \tmmathbf{u}_0$ in
  $L^2_{\sharp} [ \YY, L^2 (\Omega) ]^N$ the first part of the
  theorem (namely (\ref{eq:statementalreadyknown})) is nothing else that
  Proposition \ref{prop:g1classicalctwoscaleL2H1}. It remains to prove the
  second part. To this end let us recall (cfr. {\cite{do1976differential}})
  that since $\mathcal{M}$ is a regular closed orientable surface there exists
  an open tubular neighbourhood $U \subseteq \RR^N$ of $\mathcal{M}$ and a
  function $g : U \rightarrow \RR$ which has zero as a regular value and is
  such that $\mathcal{M}= g^{- 1} (0)$. Since $\tmmathbf{u}_{\varepsilon}
  \twoheadrightarrow \tmmathbf{u}_0$ strongly in $L^2_{\sharp} [ \YY, L^2
  (\Omega) ]^N$ we have $0 = g (\tmmathbf{u}_{\varepsilon} (x, y))
  \twoheadrightarrow g (\tmmathbf{u}_0 (x))$ strongly in $L^2_{\sharp} [
  \YY, L^2 (\Omega) ]^N$ and therefore $g (\tmmathbf{u}_0 (x)) = 0$ a.e.
  in $\Omega$. Next we observe that for any $\varepsilon \in \RR^+$ we have $g
  (\tmmathbf{u}_{\varepsilon}) = 0$ and hence $\gradx
  \tmmathbf{u}_{\varepsilon} (x, y) .\tmmathbf{n} (\tmmathbf{u}_{\varepsilon}
  (x, y)) =\tmmathbf{0}$ for a.e. $(x, y) \in \Omega \times \YY$. Passing to
  the two-scale limit we so get
  \begin{eqnarray}
    0 \eqs \int_{\Omega \times \YY} \mathcal{F}_{\varepsilon} \left( [
    \gradx \tmmathbf{u}_{\varepsilon} ] \tmmathbf{n}
    (\tmmathbf{u}_{\varepsilon}) \right) (x, y) \cdot \tmmathbf{\psi} (x, y)
    \mathd x \mathd y \qquad \qquad \qquad \qquad \qquad \qquad \qquad &  & 
    \nonumber\\
    \xrightarrow{\varepsilon \rightarrow 0} \int_{\Omega \times \YY} [
    \gradx \tmmathbf{u}_0 (x) + \grady \tmmathbf{u}_1 (x, y) ]
    \tmmathbf{n} (\tmmathbf{u}_0 (x)) \cdot \tmmathbf{\psi} (x, y) \mathd x
    \mathd y \eqs 0 &  &  \label{eq:twoscaleonmanifold}
  \end{eqnarray}
  for every $\tmmathbf{\psi} \in L^{\infty}_{\sharp} [ \YY, W (\Omega)
  ]^N$. In particular, by taking $\tmmathbf{\psi} (x, y) \assign
  \tmmathbf{\varphi} (x) \otimes 1 (y)$, since $\left\langle \grady
  \tmmathbf{u}_1 (x, y) \right\rangle_{\YY} = 0$ we have $\gradx
  \tmmathbf{u}_0 (x) \tmmathbf{n} (\tmmathbf{u}_0 (x)) =\tmmathbf{0}$ a.e. in
  $\Omega$. Thus from (\ref{eq:twoscaleonmanifold}) we get
  \[ \int_{\Omega \times \YY} \grady (\tmmathbf{u}_1 (x, y) \cdot \tmmathbf{n}
     (\tmmathbf{u}_0 (x))) \cdot \tmmathbf{\psi} (x, y) \mathd x \mathd y \eqs
     0 \quad \forall \tmmathbf{\psi} \in L^{\infty}_{\sharp} [Y, W_0
     (\Omega)]^N \]
  and hence for some $c \in \RR$ we have $\tmmathbf{u}_1 \cdot \tmmathbf{n}
  (\tmmathbf{u}_0) = c$ a.e. in $\Omega \times \YY$. But since
  $\tmmathbf{u}_1$ is null average on $Y$, so is $\tmmathbf{u}_1 \cdot
  \tmmathbf{n} (\tmmathbf{u}_0)$ and therefore necessarily $c = 0$.
\end{proof}

\begin{proof}[of Theorem \ref{thm:homogenizationhmaps}]
  For any $\tmmathbf{\psi}_{\varepsilon} \in L^{\infty}_{\sharp} [ \YY,
  W_0 (\Omega) ]^3$ we set $\tmmathbf{\eta}_{\varepsilon} \assign
  \tmmathbf{u}_{\varepsilon} \times \tmmathbf{\psi}_{\varepsilon}$ in equation
  (\ref{eq:statharmmap}). We then have $\gradx \tmmathbf{u}_{\varepsilon} 
  \gradx \tmmathbf{\eta}_{\varepsilon} = \sum_{i \in \NN_3} \partial_{x_i}
  \tmmathbf{\psi}_{\varepsilon} \cdot (\partial_{x_i}
  \tmmathbf{u}_{\varepsilon} \times \tmmathbf{u}_{\varepsilon})$ and therefore
  \begin{equation}
    \sum_{i \in \NN_3} \int_{\Omega \times \YY} a (x, y) [
    \mathcal{F}_{\varepsilon} (\tmmathbf{u}_{\varepsilon}) \times
    \mathcal{F}_{\varepsilon} \left( \frac{\partial
    \tmmathbf{u}_{\varepsilon}}{\partial x_i} \right) ] \cdot
    \frac{\partial \tmmathbf{\psi}_{\varepsilon}}{\partial x_i} \mathd x
    \mathd y = 0 \quad \forall \tmmathbf{\psi}_{\varepsilon} \in
    L^{\infty}_{\sharp} [ \YY, W_0 (\Omega) ]^3
    \label{eq:tobetwoscaled}
  \end{equation}
  By mimicking the proof of Proposition \ref{prop:attainedtest}, it is simple
  to get that for every $\tmmathbf{\eta} \in W_0 (\Omega)^3$ there exists a
  family $(\tmmathbf{\psi}_{\varepsilon})_{\varepsilon \in \RR^+}$ of
  $L^{\infty}_{\sharp} [ \YY, W_0 (\Omega) ]^3$ functions such that
  $\tmmathbf{\psi}_{\varepsilon} \twoheadrightarrow \tmmathbf{\eta}$ and
  $\gradx \tmmathbf{\psi}_{\varepsilon} \twoheadrightarrow \gradx
  \tmmathbf{\eta}$ strongly in $L^2_{\sharp} [ \YY, L^2 (\Omega)
  ]^3$, so that taking into account Proposition
  \ref{prop:g1classicalctwoscaleL2H1}, passing to the two-scale limit in
  (\ref{eq:tobetwoscaled}) we get
  \begin{equation}
    \sum_{i \in \NN_3} \int_{\Omega \times \YY} a (x, y) [ \tmmathbf{u}_0
    (x) \times \left( \frac{\partial \tmmathbf{u}_0}{\partial x_i} (x) +
    \frac{\partial \tmmathbf{u}_1}{\partial y_i} (x, y) \right) ] \cdot
    \frac{\partial \tmmathbf{\eta}}{\partial x_i} (x) \mathd x \mathd y \eqs 0
    \quad \forall \tmmathbf{\eta} \in W_0 (\Omega)^3 . \label{eq:quasihomohm}
  \end{equation}
  On the other hand, again by by mimicking the proof of Proposition
  \ref{prop:attainedtest}, we get that for every $\tmmathbf{\psi}_1 \in
  L^{\infty}_{\sharp} [Y, W_0 (\Omega)]^3$ there exists a family
  $(\tmmathbf{\psi}_{\varepsilon})_{\varepsilon \in \RR^+}$ of
  $L^{\infty}_{\sharp} [Y, W_0 (\Omega)]^3$ functions such that $\varepsilon
  \gradx \tmmathbf{\psi}_{\varepsilon} \twoheadrightarrow \grady
  \tmmathbf{\psi}_1$ strongly in $L^2_{\sharp} [ \YY, L^2 (\Omega)
  ]^3$. Hence, from Proposition \ref{prop:g1classicalctwoscaleL2H1},
  passing to the two-scale limit in (\ref{eq:tobetwoscaled}) we get
  \begin{equation}
    \sum_{i \in \NN_3} \int_{\Omega \times \YY} [ \tmmathbf{u}_0 (x)
    \times a (x, y) \left( \frac{\partial \tmmathbf{u}_0}{\partial x_i} (x) +
    \frac{\partial \tmmathbf{u}_1}{\partial y_i} (x, y) \right) ] \cdot
    \frac{\partial \tmmathbf{\psi}_1}{\partial y_i} (x, y) \mathd x \eqs 0
  \end{equation}
  for every $\tmmathbf{\psi}_1 \in L^{\infty}_{\sharp} [Y, W_0 (\Omega)]^3$.
  In particular, for any $\tmmathbf{\psi} \in L^{\infty}_{\sharp} [Y, W_0
  (\Omega)]^3$, by setting $\tmmathbf{\psi}_1 (x, y) \assign \tmmathbf{u}_0
  (x) \times \tmmathbf{\psi} (x, y)$ and taking into account that due to Lemma
  \ref{Lemma:2scalemanifold} $\tmmathbf{u}_1 (x, y) \cdot \tmmathbf{u}_0 (x) =
  0$ a.e. in $\Omega \times \YY$ we finish with the classical cell equation
  \begin{equation}
    \sum_{i \in \NN_3} \int_{\Omega \times \YY} a (x, y) \left( \frac{\partial
    \tmmathbf{u}_0}{\partial x_i} (x) + \frac{\partial
    \tmmathbf{u}_1}{\partial y_i} (x, y) \right) \cdot \frac{\partial
    \tmmathbf{\psi}}{\partial y_1} (x, y) \mathd x \eqs 0 \quad \forall
    \tmmathbf{\psi} \in L^{\infty}_{\sharp} [Y, W_0 (\Omega)]^3 .
    \label{eq:cellequationshm}
  \end{equation}
  The solution of the previous equation is classical. Indeed, due to
  Lax-Milgram lemma, the cell problem (\ref{eq:cellequationshm}), which in
  distributional form reads as
  \begin{equation}
    - \mathbf{div}_y \left( a (x, y) \grady \tmmathbf{u}_1 (x, y) \right) \eqs
    \mathbf{div}_y  (a (x, y) \nabla \tmmathbf{u}_0 (x)),
    \label{eq:cellproblemvarformharmmapdistr}
  \end{equation}
  has a unique null average solution in $L^2 [ \Omega, H^1_{\sharp}
  \left( \YY \right) ]^3$. Moreover, if for every $i \in \NN_3$ we
  denote by $\chi_{i_{}}$ the unique null average solution in $L^2 [
  \Omega, H^1_{\sharp} \left( \YY \right) ]$ of the scalar cell problem
  (\ref{eq:cellharmonic1}), by the defining the vector valued function
  $\tmmathbf{\chi} \assign (\chi_1, \chi_2, \chi_3) \in L^2 [ \Omega,
  H^1_{\sharp} \left( \YY \right) ]^3$ we get that the vector field
  \begin{equation}
    \tmmathbf{u}_1 (x, y) \assign \sum_{j \in \NN_3} \left( \tmmathbf{\chi}
    (x, y) \cdot \gradx u_0^j (x) \right) e_j  \label{eq:expru1}
  \end{equation}
  is the unique null average solution in $L^2 [ \Omega, H^1_{\sharp}
  \left( \YY \right) ]^3$ of the cell problem
  (\ref{eq:cellproblemvarformharmmapdistr}). Next we note that from
  (\ref{eq:expru1}) we get $\nabla \tmmathbf{u}_0 (x) + \grady \tmmathbf{u}_1
  (x, y) = \left( I + \grady \tmmathbf{\chi} (x, y) \right) \gradx
  \tmmathbf{u}_0 (x)$ and hence, evaluating (\ref{eq:quasihomohm}) on vector
  fields of the form $\tmmathbf{\eta} (x) \assign \tmmathbf{u}_0 (x) \times
  \tmmathbf{\varphi} (x)$ with $\tmmathbf{\varphi} \in W_0 (\Omega)^3$ and
  $\tmmathbf{\varphi} (x) \in T_{\tmmathbf{u}_0 (x)} \SStwo$ we finish with
  (\ref{eq:homogenizedharmonicmapeq}).
\end{proof}

\begin{remark}
  In general, if we do not assume any positivity condition on the coefficient
  $a$, it is not possible to reduce the domain equation (\ref{eq:quasihomohm})
  and the cell equation (\ref{eq:cellequationshm}) to a single homogenized
  equation (like the one obtained in Theorem \ref{thm:homogenizationhmaps}).
  Nevertheless the two-scale limit $\tmmathbf{u}_0$ will be a solution of the
  system of two distributional equations
  \begin{eqnarray}
    \mathbf{div}_x \left( \tmmathbf{u}_0 (x) \times \int_{\YY} a (x, y) \left(
    \gradx \tmmathbf{u}_0 (x) + \grady \tmmathbf{u}_1 (x, y) \right) \mathd y
    \right) & = & \tmmathbf{0} \quad \text{in } \mathcal{D}' (\Omega) \\
    \mathbf{div}_y  \bigl( a(x, y) \bigl( \nabla \tmmathbf{u}_0 (x) +
    \grady \tmmathbf{u}_1(x, y) \bigr) \bigr) &
    = & \tmmathbf{0} \quad \text{in } \mathcal{D}' \left( \Omega \times \YY
    \right) . 
  \end{eqnarray}
\end{remark}

\section{Conclusion and Acknowledgment}

This work was partially supported by the labex LMH through the grant no.
ANR-11-LABX-0056-LMH in the {\tmem{Programme des Investissements d'Avenir}}.

\end{document}